\documentclass[letterpaper, twoside, 11pt]{amsart}
\usepackage[T2A]{fontenc}
\usepackage[utf8]{inputenc}	
\usepackage[main=english, russian]{babel}
\usepackage{geometry}
\usepackage{comment}

\usepackage{xcolor}
\definecolor{forestgreen}{HTML}{228B22}
\definecolor{royalblue}{HTML}{4169E1}

\usepackage[unicode, pdftex]{hyperref}
\hypersetup{
	colorlinks   = true, 
	urlcolor     = blue, 
	linkcolor    = royalblue, 
	citecolor    = forestgreen 
}

\usepackage[noabbrev,capitalise,nameinlink]{cleveref}
\usepackage[abbrev,alphabetic,backrefs,initials,nobysame]{amsrefs}

\usepackage{amssymb}
\usepackage{tikz-cd}
\usepackage[all]{xy}

\newtheorem{theorem}{Theorem}[section]
\newtheorem{proposition}[theorem]{Proposition}
\newtheorem{lemma}[theorem]{Lemma}

\newtheorem{conjecture}[theorem]{Conjecture}

\theoremstyle{definition}
\newtheorem{definition}[theorem]{Definition}

\theoremstyle{definition}
\newtheorem{example}[theorem]{Example}

\theoremstyle{remark}
\newtheorem{remark}[theorem]{Remark}

\newcommand{\R}{\mathbb{R}}
\newcommand{\Ko}{\mathbb{C}}

\newcommand{\Cc}{\mathfrak{C}}
\newcommand{\Dd}{\mathfrak{D}}

\newcommand{\X}{\mathcal{X}}
\newcommand{\B}{\mathcal{B}}
\newcommand{\Z}{\mathcal{Z}}
\newcommand{\Y}{\mathcal{Y}}
\newcommand{\Hilb}{\mathbf{Hilb}_{X}}

\newcommand{\Li}{\mathcal{L}}

\newcommand{\Pp}{\mathbb{P}}
\newcommand{\Iso}{\text{Iso}}
\newcommand{\al}{\alpha}

\DeclareMathOperator{\supp}{supp}
\DeclareMathOperator{\E}{Exc}

\let\emptyset\varnothing

\title{On the finiteness of log surfaces}
\author[D. V. Serebrennikov]{Daniil Serebrennikov}
\address{Department of Mathematics, Johns Hopkins University, Baltimore, MD 21218, USA.}
\email{dserebr1@jhu.edu}
\date{}
\subjclass[2020]{14E30, 14D07}
\keywords{Minimal model program, Kawamata--Matsuki conjecture, boundedness}

\begin{document}
	\begin{abstract}
		We generalize the Kawamata–Matsuki conjecture to klt log pairs and prove it in dimension two. More precisely, we show that a log surface admits only finitely many weakly log canonical projective models with klt singularities up to log isomorphism, by reducing the problem to boundedness of polarizations.
	\end{abstract}
	\maketitle
	\tableofcontents
	
	\section{Introduction}
	Let $X$ be a projective variety over the field of complex numbers $\Ko$. We study a classical problem about the number of normal $\mathbb{Q}$-factorial projective varieties $X_\al$ such that each $X_\al$ is birational to $X$, and a canonical divisor $K_{X_\alpha}$ is nef, that is, $K_{X_\alpha}\cdot C \ge 0$ for all curves $C\subset X_\alpha$. In general, this number is not expected to be finite unless we consider only mild singularities of $X_\alpha$. For instance, one may restrict to terminal singularities (see \cref{def: sing}). In this case the varieties $X_\al$ are called \textit{minimal models} of $X$. More precisely, the following conjecture was proposed by Kawamata and Matsuki \cite[Conjecture~12.3.6]{Mat02}:
	
	\begin{conjecture}
		\label{intro-conj: classic}
		The number of projective minimal models in a fixed birational class is finite up to isomorphism.
	\end{conjecture}
	
	Let us remember that two log pairs $(X, B)$ and $(X^\prime, B^\prime)$ are said to be \textit{crepant birationally equivalent} if $X$ and $X^\prime$ are birational, and the log pullbacks of $K_X + B$ and $K_{X^\prime} + B^\prime$ coincide on any common log resolution for $(X, B)$ and $(X^\prime, B^\prime)$ (cf. \cref{def: 0-class}). Next, it is natural to extend \cref{intro-conj: classic} to a class of crepant birationally equivalent log pairs, namely a \textit{0-class}. In this setting, one needs to consider \textit{weakly log canonical models (wlc models)}. A proper log pair $(X, B)$ is a wlc model if $K_X + B$ is nef, $B$ is effective, and $(X, B)$ is lc (cf. \cref{def: wlc}).
	\begin{conjecture}
		\label{conj: finiteness}
		The number of projective wlc klt models in a fixed 0-class is finite up to log isomorphism.
	\end{conjecture}
	Note that rational projective surfaces including infinitely many $(-1)$-curves and smooth projective K3 surfaces including infinitely many $(-2)$-curves provide non-trivial examples of \cref{conj: finiteness} in dimension two (cf. \cref{ex: k3}, \cref{ex: rational}).
	\begin{theorem}
		\label{th: intro}
		\cref{conj: finiteness} holds for log surfaces.
	\end{theorem}
	Let $(X, B)$ be a projective log surface. The most delicate case of \cref{conj: finiteness} for log surfaces is when $K_X + B \equiv 0$, namely, \textit{0-pairs} (cf. \cref{def: wlc}). By definition, a $0$-pair is a wlc model. Our proof reduces the finiteness problem for 0-pairs to the boundedness of polarizations on their models. The following theorem is the main result of this work.

	\begin{theorem}
		\label{main_intro}
		Let $(X, B)$ be a projective $0$-pair of dimension 2. Consider the class $\Cc$ of all projective wlc klt models $(X_\alpha,B_\alpha)$ that are crepant birationally equivalent to $(X,B)$. Let $\Dd$ be the class of all varieties $X_\alpha$ corresponding to the pairs $(X_\alpha,B_\alpha)$ in~$\Cc$.
		If the class $\Dd$ has bounded polarization, then the set of (log) isomorphism classes in $\Cc$ is finite.
	\end{theorem}
	
	\begin{remark}
		The restriction on singularities in \cref{conj: finiteness} cannot be weakened to log canonical singularities. Indeed, consider the log pair $(X, B) = (\Pp^2_\Ko, C)$, where $C\subset \Pp^2_\Ko$ is a smooth cubic curve. Let $q_1: X_1\to X$ be the blow-up of a point $p_1\in B$, and define $B_1$ by the equation $K_{X_1} + B_1 = q_1^*(K_X + B)$. By construction, $(X_1, B_1)$ is crepant birationally equivalent to $(X, B)$. Note that $B_1 = (q_1^{-1})_*B + a_1E_1$ for some $a_1\in \mathbb{Q}$, where $E_1$ is the exceptional divisor for $q_1$. It follows from the adjunction and projection formulas that $a_1 = 0$. Similarly, the blow-up of any point $p_2 \in B_1 \smallsetminus E_1$ adds to the boundary a divisor $a_2E_2$ with $a_2 = 0$. Repeating this process, we obtain infinitely many pairwise non-isomorphic projective wlc models $(X_i, B_i)$ in the 0-class of~$(X, B)$.
	\end{remark}

	\addtocontents{toc}{\protect\setcounter{tocdepth}{1}}
	\subsection*{Acknowledgments}
	The author is grateful to Professor Shokurov for the ideas underlying this work, guidance, and support. The author thanks the anonymous referee for valuable comments that improved the exposition.
	
	\section{Preliminaries}
	Throughout, we work over the field of complex numbers $\mathbb{C}$ and follow the standard terminology \cite{IS05} unless stated otherwise. By a variety we mean an integral separated scheme of finite type over $\mathbb{C}$. On a normal variety, a prime divisor is understood in the Weil sense, i.e. an integral closed subscheme of codimension~$1$. In every part of the paper, for an $\mathbb{R}$-divisor $D$ we write $D=\sum_i d_i D_i$, where $D_i$ are distinct prime divisors and $d_i\in\mathbb{R}$.
	
	A \textit{pair} $(X,B)$ consists of a normal variety $X$ and an $\mathbb{R}$-divisor $B$ on $X$. We say that $(X,B)$ is a \textit{log pair} if $K_X+B$ is an $\mathbb{R}$-Cartier $\R$-divisor. Suppose $B=\sum_i b_i B_i$, and $\Gamma\subseteq\mathbb{R}$ is a set. If $b_i\in\Gamma$ for all~$i$ then we write $B\in\Gamma$. Denote by $[0,1]$ the unit segment of real numbers.  The $\mathbb{R}$-divisor $B$ is called a \textit{boundary} if $B\in [0,1]$. A pair $(X,B)$ is called \textit{projective} (resp. \textit{of dimension} $d$) if $X$ is projective (resp. of dimension $d$). The pairs $(X,B)$ and $(X',B')$ are said to be \textit{log isomorphic} if there exists an isomorphism $\varphi:X\to X'$ such that $\varphi_*(B)=B'$; we then write $(X,B)\approx (X',B')$. We use the notation $\cong$ instead of $\approx$ if the isomorphism is natural.
	
	\begin{definition}
		\label{def: sing}
		Let $(X,B)$ be a log pair with $B=\sum b_i B_i$. We say that $(X,B)$ has the following singularities when the corresponding inequalities hold:
		\[
		\begin{gathered}
			\begin{cases}
				\text{terminal (trm)} \\
				\text{canonical} \\
				\text{Kawamata log terminal (klt)}\\
				\epsilon\text{-Kawamata log terminal ($\epsilon$-klt)}\\
				\text{log canonical (lc)}
			\end{cases}
			\Longleftrightarrow
			\text{dis}(X,B)
			\begin{cases}
				>0;\\
				\ge 0;\\
				>-1 \ \text{and}\  b_i<1 \ \forall i;\\
				>-1+\epsilon \ \text{and}\  b_i<1-\epsilon \ \forall i;\\
				\ge -1 \ \text{and}\  b_i\le 1 \ \forall i.
			\end{cases}
		\end{gathered}
		\]
		Here, the \textit{discrepancy} of $(X, B)$ is given by $$\text{dis}(X, B) = \underset{E}{\inf}\{a(E; X, B): E \text{ is an exceptional prime divisor over } X\},$$ that is, $E$ runs through all the prime exceptional divisors of all birational morphisms $q: X\to Y$, and the number $a(E; X, B)\in \mathbb{R}$ is defined by the formula $K_Y = q^*(K_X + B) + \sum_E a(E; X, B)\, E$.
	\end{definition}
	
	\begin{definition}
		\label{def: wlc}
		A log pair $(X,B)$ is called a \textit{weakly log canonical model (wlc model)} if the following hold:
		\begin{itemize}
			\item $X$ is a proper variety, and $B$ is a boundary.
			\item $(X,B)$ has log canonical singularities.
			\item $K_X+B$ is nef.
		\end{itemize}
		If $(X,B)$ has klt (resp. terminal) singularities, then $(X,B)$ is called a \textit{wlc klt model} (resp. a \textit{wlc trm\footnote{The abbreviation "trm" is introduced in \cite{IS05}, and we use it systematically in this paper.} model}). If, in addition, $K_X+B\equiv 0$, then $(X,B)$ is called a \textit{0-pair}\footnote{In the literature, 0-pairs are also called Calabi–-Yau pairs.}.
	\end{definition}
	
	\begin{definition}
		\label{def: 0-class}
		We say that a log pair $(X_\alpha,B_\alpha)$ is \textit{crepant birationally equivalent} to $(X,B)$ if the varieties  $X_\alpha$ and $X$ are birationally equivalent, and for each common log resolution $Y$ of these pairs there is an $\R$-divisor $D$ on $Y$ such that the following equalities hold:
		\[
		\begin{gathered}
			K_Y+D=q^*(K_X+B)=q_\alpha^*(K_{X_\alpha}+B_\alpha),\\
			B=q_*D,\quad q_\alpha=(q_\alpha)_*D,
		\end{gathered}
		\]
		where $q:Y\to X$ and $q_\alpha:Y\to X_\alpha$ are the corresponding log resolutions. The class of all log pairs crepant birationally equivalent to $(X,B)$ is called the \textit{0-class}\footnote{Assuming the abundance conjecture, every wlc model is a "relative 0-pair" over its log canonical model, which is unique. So, a 0-class is a version of a birational class for relative 0-pairs.} of the log pair $(X,B)$.
	\end{definition}
	
	\begin{remark}
		If $(X,B)$ is a wlc klt model (resp. 0-pair), then each log pair $(X_\alpha,B_\alpha)$ crepant birationally equivalent to $(X, B)$ with effective $\mathbb{R}$-divisor $B_\alpha$ is also a wlc klt model (resp. 0-pair).
	\end{remark}
	
	\begin{example}
		\label{ex: k3}
		Let $X$ be a smooth projective K3 surface $(K_X \sim 0)$ including infinitely many distinct $(-2)$-curves $C_i\subset X$, that is, smooth rational curves with $C_i^2=-2$. For any rational number $0< \epsilon < 1$, the log pair $(X, \epsilon C_i)$ is klt, and $(K_X + \epsilon C_i)\cdot C_i = -2\epsilon < 0$. Hence, there exists a projective birational morphism $q_i:X\to X_i$ contracting $C_i$ such that $(X_i,0)$ is klt due to \cite[Theorem 3.7]{KM98}. Moreover, the adjunction and projection formulas imply that $K_{X_i} \sim 0$. Hence, each log pair $(X_i, 0)$ is a projective klt $0$-pair in the 0-class of $(X,0)$.
	\end{example}
	
	\begin{example}
		\label{ex: rational}
		Let $C_1 \subset \Pp^2_\Ko, C_2\subset\Pp^2_\Ko$ be smooth cubic curves intersecting in distinct points $p_1, p_2, \dots, p_9$. In addition, we assume that these points are non-torsion on both curves $C_1, C_2$. Set $q: \textnormal{Bl}_{p_1, \dots p_9}\Pp^2_\Ko \to \Pp^2_\Ko$, and $\widetilde{C_i} = q_*^{-1}C_i$. Then the pair $(X, B) = (\textnormal{Bl}_{p_1, \dots p_9}\Pp^2_\Ko, \frac{1}{2}\widetilde{C_1} + \frac{1}{2}\widetilde{C_2})$ is a klt $0$-pair that contains infinitely many $(-1)$-curves $E_i\subset X$. By the Cone Theorem \cite[Theorem 3.7]{KM98}, there exists a projective birational morphism $q_i: X\to X_i$ contracting the curve $E_i$. Then each log pair $(X_i, (q_i)_*B)$ is a projective klt $0$-pair in the $0$-class of $(X, B)$.
	\end{example}
	
	Let $f:\X\to S$ be a morphism of schemes and $s\in S$ a point. We denote the scheme-theoretic fiber by $\X_s=\X\times_S s$. If $\Li$ is a line bundle (an invertible sheaf) on $\X$, then its restriction to the fiber $\X_s$ is denoted by $\Li_s=\Li|_{\X_s}$. When no ambiguity arises, we denote a morphism $f: \X \to S$ simply by $\X/S$.
	
	\begin{definition}
		Suppose $f: \X \to S$ is a proper flat morphism between varieties, and the base $S$ is non-singular. We say that $\X/S$ is a \textit{family of varieties} if every fiber $\X_s$ is a variety. If, in addition, the morphism $f$ is projective, then $\X/S$ is called a \textit{projective family of varieties}.
	\end{definition}
	
	Consider a family of varieties $\X/S$ and an $\mathbb{R}$-divisor $\B=\sum b_i\B_i$ on $\X$. In general, the restriction of $\B$ to a fiber $\X_s$ is not defined, and even when it is, the definition is not straightforward. For our purposes it is convenient to work with \textit{elementary families} on which $\B_s$ is defined.
	
	\begin{definition}
		\label{def: elementary_family}
		Let $(\X,\B)$ be a pair and $f:\X\to S$ be a family of varieties. Suppose $\B=\sum_{i=1}^n b_i\B_i$ is a decomposition of the $\mathbb{R}$-divisor into distinct prime divisors. We say that $(\X/S,\B)$ is an \textit{elementary family} if for every point $s\in S$ the following hold:
		\begin{enumerate}
			\item The fiber $\X_s$ is a normal variety.
			\item The restriction $f|_{\B_i}:\B_i\to S$ is flat for all $i$.
			\item The closed subscheme $(\B_i)_s=\B_i\times_S s$ is a prime divisor on $\X_s$ for all $i$.
			\item The prime divisors $(\B_i)_s$ and $(\B_j)_s$ are distinct whenever $i\ne j$.
		\end{enumerate}
		In this situation, we put $\B_s=\sum_i b_i(\B_i)_s$. If, in addition, the morphism $f$ is projective, then $(\X/S,\B)$ is called a \textit{projective elementary family}.
	\end{definition}
	
	\begin{definition}
		\label{def: bounded_varieties}
		A class $\Dd$ of projective varieties $X_\alpha$ is called \textit{bounded} if there exist finitely many projective families $\X^{(j)}/S^{(j)}$ such that:
		\begin{enumerate}
			\item For every $X_\alpha$ in $\Dd$ there is an index $j$ and a closed point $s_\alpha\in S^{(j)}$ with $\X_{s_\alpha}^{(j)}\approx X_\alpha$.
			\item For each index $j$ and every closed point $s\in S^{(j)}$ there exists $X_\alpha$ in $\Dd$ with $\X_s^{(j)}\approx X_\alpha$.
		\end{enumerate}
	\end{definition}
	
	\begin{definition}
		\label{def: bounded_pairs}
		A class $\Cc$ of projective pairs $(X_\alpha,B_\alpha)$ is called \textit{bounded} if there exist finitely many projective elementary families $(\X^{(j)}/S^{(j)},\B^{(j)})$ such that:
		\begin{enumerate}
			\item For every $(X_\alpha,B_\alpha)$ in $\Cc$ there is an index $j$ and a closed point $s_\alpha\in S^{(j)}$ such that $(X_\alpha,B_\alpha)\approx~(\X^{(j)}_{s_\alpha},\B^{(j)}_{s_\alpha}).$
			\item For each index $j$ and every closed point $s\in S^{(j)}$ there exists a pair $(X_\alpha,B_\alpha)$ in $\Cc$ such that $(\X^{(j)}_s,\B^{(j)}_s)\approx~(X_\alpha,B_\alpha).$
		\end{enumerate}
	\end{definition}

	\begin{definition}
		A class $\Cc$ of projective pairs $(X_\alpha,B_\alpha)$ \textit{has bounded polarization} if there exists a positive integer $N\in \mathbb{Z}_{>0}$ and a finite subset $\Gamma\subset\mathbb{R}$ such that $B_\alpha\in\Gamma$, and each $X_\alpha$ admits a very ample Cartier divisor $H_\alpha$ satisfying the following inequalities:
		\[
		(B_\alpha)_{\mathrm{red}}\cdot H_\alpha^{\dim X_\alpha-1}\le N,\qquad
		H_\alpha^{\dim X_\alpha}\le N,
		\]
		where $(B_\alpha)_{\mathrm{red}}=\sum_{V\in\supp(B_\alpha)} V$, $\ \supp(B_\alpha)=\{\,B_{i\alpha}\mid b_{i\alpha}\ne 0\,\}$, and $B_\alpha=\sum_i b_{i\alpha}B_{i\alpha}$.
		
		Similarly, a class of projective (not necessarily normal) varieties $X_\alpha$ has bounded polarization if there exist a positive integer $N\in \mathbb{Z}_{>0}$ and very ample Cartier divisors $H_\alpha$ on $X_\alpha$  such that $H_\alpha^{\dim X_\alpha}\le N$.
	\end{definition}
	
	\begin{remark}
		\label{remark: components}
		As the intersection $(B_\alpha)_{\mathrm{red}}\cdot H_\alpha^{\dim X_\alpha-1}$ is bounded, the set $\{\#\supp(B_\alpha)\}\subseteq \mathbb{Z}_{\ge0}$ is also bounded. We write $\# A$ for the number of elements of a finite set $A$.
	\end{remark}
	For ease of reference, we recall the following general result proven by Grothendieck.
	\begin{proposition}\cite[Appendix~E]{GW20}
		\label{prop: open_properties}
		Let $\X \to S$ be a flat proper morphism between Noetherian schemes. Then the following sets are open in $S$:
		\begin{itemize}
			\item
			$\{s\in S: \X_s \ \text{is geometrically integral} \  \}$;
			\item 
			$\{s\in S: \X_s \ \text{is geometrically normal} \ \}$.
		\end{itemize}
	\end{proposition}
	For the reader’s convenience, we prove the following statement, which often appears in the literature without proof.
	\begin{lemma}
		\label{lemma: H-boundedness(varieties)}
		Let $\Dd$ be a class of projective varieties of fixed dimension $d\in \mathbb{Z}_{>0}$. If $\Dd$ has bounded polarization, then $\Dd$ is a subclass of a bounded class of projective varieties.
	\end{lemma}
	
	\begin{proof}
		We construct a projective morphism $f:\X\to S$ between quasi-projective schemes such that for every variety $X_\al$ in $\Dd$ there exists a closed point $s_\al\in S$ together with a closed embedding $\varphi_\alpha: X_\alpha \hookrightarrow \Pp^{N+d}_\mathbb{C}$ such that $\varphi_\alpha(X_\alpha)=\X_{s_\alpha}$. By the definition of a class with bounded polarization, there exists a positive integer $N$ and very ample Cartier divisors $H_\alpha$ on $X_\al$ such that $H_\al^{d}\le N$. From Matsusaka’s inequalities \cite[Chapter~VI, Exercise~2.15.8.5]{Kol96} it follows that
		\[
		\dim_\mathbb{C} H^0(X_\al,H_\al)\le H_\al^{d}+d\le N+d<+\infty.
		\]
		Hence for every $X_\al\in\Dd$ there is an embedding $X_\al\hookrightarrow \Pp^{\dim|H_\al|}_\mathbb{C}\subseteq \Pp^{N+d}_\mathbb{C}$ such that $\deg(X_\al)=(H_\alpha)^d\le N$. According to \cite[Chapter~I, Theorem~3.21]{Kol96} for every positive integer $N^\prime\le N$ there is a universal family
		\[
		\mathbf{Univ}_{d,N^\prime}(\Pp^{N+d}_\mathbb{C}) \longrightarrow \mathbf{Chow}_{d,N^\prime}(\Pp^{N+d}_\mathbb{C})
		\]
		parametrizing $d$-dimensional effective cycles of degree $N^\prime$ in $\Pp^{N+d}_\mathbb{C}$. Define the following schemes: $$\mathbf{Univ}_{d,\le N}(\Pp^{N+d}_\mathbb{C}) = \bigsqcup_{1\le N^\prime \le N} \mathbf{Univ}_{d,N^\prime}(\Pp^{N+d}_\mathbb{C}), \ \mathbf{Chow}_{d,\le N}(\Pp^{N+d}_\mathbb{C}) = \bigsqcup_{1\le N^\prime\le N} \mathbf{Chow}_{d, N^\prime}(\Pp^{N+d}_\mathbb{C}).$$
		Let $S$ be the closure of the set $\{[X_\alpha]\}\subseteq \mathbf{Chow}_{d,\le N}(\Pp^{N+d}_\mathbb{C})$, endowed with the structure of a reduced closed subscheme. Then the projection
		\[
		f:\ \X=\mathbf{Univ}_{d,\le N}(\Pp^{N+d}_\Ko)\times_{\mathbf{Chow}_{d,\le N}(\Pp^{N+d}_\Ko)} S \longrightarrow S
		\]
		is a projective morphism.
		
		Since the base is of finite type over $\Ko$, we may assume that $S$ is irreducible. By the generic flatness theorem \cite[Theorem~5.12]{FGA05} for a morphism of finite type (over an integral Noetherian scheme), there are finitely many locally closed subschemes $S_i \subseteq S$ such that the restriction $\X\times_S S_i\to S_i$ of $f$ is flat, and $\bigcup_i S_i = S$. After the base change $S^\prime = \bigsqcup_i S_i\to S$ we obtain a flat projective morphism $f^\prime: \X^\prime = \X\times_S S^\prime \to S^\prime$. To ease the notation, we continue to write $f: \X\to S$ for the corresponding new family. By \cref{prop: open_properties}, the set $\{\,s\in S:\, \X_s\ \text{is geometrically integral}\,\}$ is open in $S$. Therefore, after shrinking the base, we may assume that $\X_s$ is a variety for every closed point $s\in S$. As the singular locus of $S$ is closed, after a further stratification and base change we may assume that $S$ is non-singular. Set $\X_j=\X\times_S S_j$, where $\{S_j\}$ is the finite set of irreducible components of the base. Then each $\X_j/S_j$ is a projective family of varieties, and every variety $X_\al$ in $\Dd$ is isomorphic to some fiber of one of these families.
	\end{proof}
	
	\begin{remark}
		In general, a subclass of a bounded class of projective varieties need not be bounded in the sense of \cref{def: bounded_varieties}. For example, consider the class of all elliptic curves with rational $j$-invariant.
	\end{remark}
	
	At the beginning of the main section (\cref{sec: main}), we prove a version of \cref{lemma: H-boundedness(varieties)} for pairs. Its proof is more technical than the proof of \cref{lemma: H-boundedness(varieties)} because it requires extra care regarding conditions (3) and (4) of \cref{def: elementary_family}.
	\begin{proposition}[{\cref{lemma: H-boundedness}}]
		Let $\Cc$ be a class of projective pairs of fixed dimension $d \in \mathbb{Z}_{>0}$. If $\Cc$ has bounded polarization, then $\Cc$ is a subclass of a bounded class of projective pairs.
	\end{proposition}
	
	\section{Isotrivial families}
	\label{sec: isotriviality}
	In this section, we work with pairs and families of pairs. To distinguish between them, we use ordinary letters $(X, B), (Y, D), \text{etc.}$ for pairs and calligraphic letters $(\X/S, \B), (\Y/S, \mathcal{D}), \text{etc.}$ for families. Let us remark that these two notions coincide When the base $S$ is a single point.
	
	Let $(\X/S, \B)$ be a projective elementary family, and $(X,B)$ be a pair. Define the sets
	\[
	\text{Iso}_X(\X/S) =  \{s\in S: \X_s\approx X\}, \quad
	\text{Iso}_{(X, B)}(\X/S, \B) =  \{s\in S: (\X_s, \B_s)\approx (X, B)\}.
	\]
	
	From now and on, we fix a log pair $(X, B)$ of dimension 2, namely a \textit{log surface}. Suppose that $(X, B)$ is a projective trm 0-pair. Since $X$ has only terminal singularities, $X$ is a smooth surface \cite[Theorem~4.5]{KM98}. We also fix a projective elementary family $(\X/S, \B)$ such that
	\begin{itemize}
		\item $f: \X \to S$ is a smooth projective family of varieties.
		\item $(\X, \B)$ is a terminal log pair.
		\item $K_\X + \B \sim_{S, \mathbb{R}} 0$, that is, $K_\X + \B \sim_{\mathbb{R}} f^*(L)$ for some $\R$-Cartier $\R$-divisor $L$ on $S$.
		\item $\text{Iso}_{(X, B)}(\X/S, \B)$ is (Zariski) dense in $S$.
	\end{itemize}
	In this section, we show that $(\X/S, \B)$ is an isotrivial family over an open dense subset of $S$. The general strategy of the proof is inspired by \cite[Theorem~4.12]{Xu25}, where an analogous statement is proved for certain 0-pairs in dimension 3 with boundary having rational coefficients.
	
	\begin{definition}
		We say that a morphism $\tau: U\to S$ between non-singular varieties is an \textit{\'etale base change} if the image $\tau(U)$ is an open dense subset of $S$, and $\tau: U\to \tau(U)$ is an \'etale covering, that is, an \'etale, finite, surjective morphism. For an elementary family $(\X/S, \B)$ with $\B = \sum b_i \B_i$ we set $(\X_U, \B_U) = (\X, \B)\times_S U = (\X\times_S U, \sum b_i (\B_i\times_S U))$. It is clear that $(\X_U/U, \B_U)\times_S U$ is an elementary family over~$U$.
	\end{definition}
	
	\begin{theorem}
		\label{th: iso_pairs}
		There is an \'etale base change $U\to S$ such that $(\X, \B)\times_S U \approx (X, B)\times_\Ko U$ over~$U$.
	\end{theorem}
	
	\begin{proof}
		\textbf{Step 1.}\label{case1} The boundary $\B$ is a $\mathbb{Q}$-divisor. We introduce the setup from \cite[Section~2]{Amb05} and adapt it to our situation. For an $\mathbb{R}$-divisor $D = \sum_i d_iD_i$ we define $\lfloor D\rfloor = \sum_i\lfloor d_i \rfloor D_i$, and $\{D\} = D - \lfloor D \rfloor$. Note that the variety $\X$ is non-singular, as $\X/S$ is a smooth family with non-singular base $S$. To be definite, we put $(X, B) = (\X_{s_0}, \B_{s_0})$ for some point $s_0\in S$.
		\begin{enumerate}
			\item Let $\mu: \Y \to \X$ be an equivariant resolution of $\X$ with respect to $\B_\text{red}$  (cf. \cite[Lemma~1.1]{Amb05}), $K_\mathcal{Y} + \mathcal{D} = \mu^*(K_\X + \B)$ is the log pullback, and $\mathcal{E} = \{\mathcal{D}\}_\text{red}$. We can assume that the log pair  $(\Y, \mathcal{D})$ is log smooth over an open dense subset in $S$ by generic smoothness. 
			\item \label{property_2}
			 Let $b$ be the minimal positive integer such that  $b\mathcal{D}$ is integral and $b(K_\Y + \mathcal{D}) \sim 0$ over the generic point of $S$. Choose a rational function $\varphi\in k(\Y)$ with $b(K_\Y + \mathcal{D}) = \text{div}_\Y(\varphi)$. Let $\pi: \widetilde{\mathcal{Y}} \to \Y$ be a normalization of $\mathcal{Y}$ in $k(\Y)(\sqrt[b]{\varphi})$, that is, an index-one cover.
			\item Let $\nu: \mathcal{V} \to \widetilde{\Y}$ be an equivariant resolution with respect to $\text{Sing}(\widetilde{\Y})$, i.e. the singular locus  of $\widetilde{\Y}$. Let $h: \mathcal{V} \to S$ be the induced (projective) morphism. Without loss of generality, it can be assumed that the morphism $h$ is smooth. Since $h$ is projective, there exists an ample over $S$ line bundle $\Li$ on $\mathcal{V}$, that is, $\Li_s$ is ample for every $s\in S$ \cite[Theorem 1.7.8]{Laz04}.
			\item Let $\mathbb{D}$ be the period domain, and let $\Gamma$ be the monodromy group associated with the variation of integral polarized Hodge structures $({R}^2h_*\mathbb{Z}_\mathcal{V})_\text{prim}\otimes \mathcal{O}_{S}$\footnote{Slightly abusing the notation, we denote the complex manifold associated to $S$ by the same letter.}. Define $\Phi: S\to \Gamma \backslash \mathbb{D}$ to be the associated period map, which is an analytic map between analytic spaces (see \cref{app}).
			\item \label{property_5}
			Let $\kappa_{s}: T_{S, s} \to H^1(\Y_s, T_{\Y_s}\left(-\log \mathcal{E}_s\right))$ and $\kappa_{\mathcal{V}_s}: T_{S, s} \to H^1(\mathcal{V}_s, T_{\mathcal{V}_s})$ be the induced Kodaira--Spencer classes. These maps have the same kernel for general closed points $s\in S$, equal to $T_{\Phi^{-1}(\Phi(s)), s}\subseteq T_{S, s}$ \cite[Proposition~2.1]{Amb05}.
		\end{enumerate}
		
		\[
		\xymatrix{
			(\Y, \mathcal{D}) \ar[d]_\mu & \widetilde{\Y} \ar[l]_(0.4){\pi} & {\mathcal{V}} \ar[l]_(0.4){{\nu}} \ar[lldd]^h	\\
			(\X, \B) \ar[d]_f	\\
			S & &
		}
		\]
		
		\begin{lemma}
			\label{lemma: iso_Q_divisor}
			Assume that $\textnormal{Ker}(\kappa_s) = 0$ for general points $s\in S$. Then the base $S$ is a single point.
		\end{lemma}
		
		\begin{proof}
			 Recall that $(X, B) = (\X_{s_0}, \B_{s_0})$ for some point $s_0\in S$. Since $(\X/S, \B)$ is an elementary family, we may assume that $(\Y/S, \mathcal{D})$ is a projective elementary family after shrinking the base. After passing to an open dense subset, we can assume that the induced morphism $\mu_s: \Y_s \to \X_s$ is an equivariant resolution with respect to $(\B_s)_\text{red}$ for every $s\in S$. Hence $\text{Iso}_{(X, B)}(\X/S, \B)\subseteq \text{Iso}_{(\Y_{s_0}, \mathcal{D}_{s_0})}(\Y/S, \mathcal{D})$.
			
			Now we show that $\text{Iso}_{(\Y_{s_0}, \mathcal{D}_{s_0})}(\Y/S, \mathcal{D}) \subseteq \text{Iso}_{\widetilde{\Y}_{s_0}}(\widetilde{\Y})$ after passing to an appropriate open dense subset in $S$. Recall that $b$ is the minimal positive integer such that $b(K_\Y + \mathcal{D}) \sim 0$ over the generic point of $S$. Since $\mathcal{Y}/S$ is a smooth family, we have $b(K_\Y + \mathcal{D})|_{\Y_s} = b(K_{\Y_s} + \mathcal{D}_s)$ for every point $s\in S$. By the semicontinuity theorem \cite[Theorem~12.8]{Har77}, the set $$I_{b^\prime} = \{s\in S: h^0\left(\Y_s, \mathcal{O}_{\Y_s}(b^\prime(K_{\Y_s} + \mathcal{D}_s))\right) = 1\}$$ is locally closed (constructible) for every positive integer $b^\prime\le b$. Moreover, $I_{b^\prime}$ contains the generic point if and only if $b^\prime = b$. Hence after shrinking the base, we can assume that $b$ is the minimal positive integer such that $b(K_{\Y_s} + \mathcal{D}_s) \sim 0$ for every point $s\in S$, and $b(K_{\Y_s} + \mathcal{D}_s) = \text{div}_{\Y_s}(\varphi_s)$, where $\varphi_s$ is the restriction of the rational function $\varphi\in k(\Y)$ to $\Y_s$.
			By the flat stratification theorem, the induced morphism $\widetilde{\Y}\to S$ is flat over an open dense subset. By \cref{prop: open_properties}, the variety $\widetilde{\mathcal{Y}}_s$ is normal for every $s\in S$ after replacing the base by an open dense subset. Then, the variety $\widetilde{\Y}_s$ is an index-one cover of $(\Y_s, \mathcal{D}_s)$ for every point $s\in S$ by construction. It follows that $\text{Iso}_{(\Y_{s_0}, \mathcal{D}_{s_0})}(\Y/S, \mathcal{D}) \subseteq \text{Iso}_{\widetilde{\Y}_{s_0}}(\widetilde{\Y})$ because an index-one cover is unique up to isomorphism.
			
			Next, as $(\Y/S, \mathcal{D})$ is an elementary projective family, each irreducible component of $\text{Sing}(\widetilde{\Y})$ is horizontal over an open dense subset in $S$. We may shrink $S$, so that $\mathcal{V}_s\to \widetilde{\Y}_s$ is an equivariant resolution of singularities with respect to $\text{Sing}(\widetilde{\Y}_s)$ for every $s\in S$. This yields $\text{Iso}_{\widetilde{\Y}_{s_0}}(\widetilde{\Y}/S) \subseteq \text{Iso}_{\mathcal{V}_{s_0}}(\mathcal{V}/S)$. This implies that the set $\text{Iso}_{\mathcal{V}_{s_0}}(\mathcal{V}/S)$ is dense in $S$, because it includes $\text{Iso}_{(X, B)}(\X/S, \B)$.
			
			Finally, we use the condition $\text{Ker}(\kappa_s) = 0$, density of $\text{Iso}_{\mathcal{V}_{s_0}}(\mathcal{V}/S)$ together with global properties of the period map $\Phi$ to prove the lemma. According to \cref{lem: finiteness_polarizations}, the set $P:= \Phi (\text{Iso}_{\mathcal{V}_{s_0}}(\mathcal{V}/S))$ is finite. Let $\overline{S}\supseteq S$ be a non-singular compactification with simple normal crossing boundary. By \cite[Proof of Theorem 9.6 (p.158)]{Gri70}, the period map $\Phi$ extends to a proper analytic morphism $\overline{\Phi}: S'\to \Gamma \backslash \mathbb{D}$ defined on non-empty open subset $S'\subseteq \overline{S}$ across the locus of finite monodromy. Hence, $\overline{\Phi}^{-1}(P)$ is a closed analytic subset of $S'$, and $\overline{\Phi}^{-1}(P)$ is compact. Eventually, $\overline{\Phi}^{-1}(P)$ is a closed algebraic subset by \cite[Theorem 1.1]{BBT23}. Note that $\overline{\Phi}^{\,-1}(P)$ is (Zariski) dense in $S'$ because the former includes $\text{Iso}_{\mathcal{V}_{s_0}}(\mathcal{V}/S)$. Therefore, $\overline{\Phi}^{\,-1}(P) = S'$, and ${\Phi}^{-1}(\text{P}) = S$. In addition, $P$ is an one-point set as $S$ is irreducible. Shrinking the base, we can assume that $\text{Ker}(\kappa_s) = 0$ for every closed point $s\in S$ due to the hypothesis. Hence, for all $s\in S$ we have that $\text{Ker}(\kappa_s) = T_{S, s}$ according to \hyperref[property_5]{Property 5}. Then the assumption $\text{Ker}(\kappa_{\mathcal{V}_s}) = 0$ implies that $T_{S, s} = \{0\}$ for all $s\in S$. Since $S$ is irreducible, we conclude that $S = \{s_0\}$, as claimed.
		\end{proof}

		To conclude the proof of the first case we apply the following theorem (cf. \cite[Lemma 7.1]{Kaw85}):   
		\begin{theorem}\cite[Theorem 2.2]{Amb05}
			\label{th: Ambro}
			There exist dominant morphisms $\tau : \overline{S} \to S$ and $\varrho: \overline{S}\to S^!$ with $\tau$ generically finite and $\overline{S}, S^!$ non-singular, and there exist a log pair $(\X^!, \B^!)$ with $K_{\X^!} + \B^! \sim_{\mathbb{Q}} 0$ and a projective contraction $f^!: \X^!\to S^!$ satisfying the following properties:
			\begin{enumerate}
				\item 
				there exists an open dense subset $U\subseteq \overline{S}$ and a log isomorphism $(\X, \B)\times_S U \approx (\X^!, \B^!)\times_{S^!} U$ over~$U$;
				\item 
				The Kodaira--Spencer map $\kappa_{s^!}$ is injective for general points $s^!\in S^!$. 
			\end{enumerate}
			\[
			\xymatrix{
				(\X, \B)\ar[dd]_f&&(\X^!, \B^!)\ar[dd]^{f^!}\\
				&\overline{S}\ar[ld]_\tau \ar[rd]^\varrho&\\
				S&&S^!
			}
			\]
		\end{theorem}
	By \cref{lemma: iso_Q_divisor}, we obtain $(\X^!, \B^!) = (X, B)$ and $S^! = \text{Spec}\,\Ko$. It follows from \cref{th: Ambro} and its proof that, after shrinking the base and passing to an \'etale covering $U\to S$, there is a log isomorphism $(\X, \B)\times_S U \approx (X, B)\times_\Ko U$ over~$U$, as required. Before proving the second case, we introduce some definitions.

	\begin{definition}{\cite[Definition~2.6.9]{Ser06}}
		\label{def: isotrivial}
		Let $\X/S$ be a family of varieties.
		\begin{enumerate}
				\item $\X/S$ is called a \textit{trivial family} if there exists a variety $X$ and an isomorphism $X\times_\Ko S\approx\X$ over $S$.
				\item $\X/S$ is called an \textit{isotrivial family} if there exists an \'etale covering $U\to S$ such that $\X\times_S U\to U$ is a trivial family.
				\item $\X/S$ is called a \textit{locally isotrivial family} if every closed point $s\in S$ has a neighborhood $U$ (in the Zariski topology) such that the family $\X\times_S U\to U$ is isotrivial.
		\end{enumerate}
	\end{definition}
		
	\begin{proposition}{\cite[Proposition~2.6.10]{Ser06}}
		\label{prop: triv_def}
			Let $\X/S$ be a projective family of varieties and $X$ a projective variety. If $\textnormal{Iso}_{X}(\X/S)=\{\text{all closed points of } S\}$, then the family $\X/S$ is locally isotrivial.
	\end{proposition}

	\textbf{Step 2.} The boundary $\B$ is an $\mathbb{R}$-divisor. First, we show that the family $\X/S$ is isotrivial (see \cref{def: isotrivial} and \cref{prop: triv_def}) after shrinking the base. Recall that $(X, B) = (\X_{s_0}, \B_{s_0})$ for some point $s_0\in S$. In accordance with \cite[Proposition~3.2(3)]{Bir11}, the $\mathbb{R}$-divisor $\B$ can be written as an $\mathbb{R}$-linear combination of effective $\mathbb{Q}$-divisors $\Delta_1, \Delta_2, \dots, \Delta_I$ with coefficients $r_1, r_2, \dots, r_I\in \mathbb{R}$ respectively such that $K_\X + \Delta_i \sim_{S, \mathbb{Q}} 0$ for all $i = 1, 2, \dots, I$, and $\sum_{i=1}^I r_i = 1$. We may assume that all coefficients $r_i$ are different. Indeed, suppose that $r_{i_1} = r_{i_2} = \dots = r_{i_k}:= r$ for some indices. Define $\Delta^\prime  = \frac{1}{k}(\Delta_{i_1}+ \Delta_{i_2} + \dots + \Delta_{i_k})$ and $r^\prime = rk$. Then $r(\Delta_{i_1}+ \Delta_{i_2} + \dots + \Delta_{i_k}) = r^\prime \Delta^\prime$ and $K_\X + \Delta^\prime \sim_{S, \mathbb{Q}} 0$. Note we can assume that there is at least one index $i_0$ for which the coefficient $r_{i_0}$ is irrational (otherwise $\B$ is a $\mathbb{Q}$-divisor). Furthermore, we can assume that $I$ is the minimum possible number, all $\Delta_i$ are affinely independent over $\mathbb{Q}$, and the coefficients $r_i$ are linearly independent over $\mathbb{Q}$ in accordance with the convex approximation for $\R$-divisors (see \cite[Proposition 11.46]{Kol23}). In particular, $r_{i_0}\notin \text{Span}_\mathbb{Q}\{r_i: i\neq i_0\}$. This implies the following inclusion $$\Iso_{(X, B)}(\X/S, \B)\subseteq \Iso_{\left(X, (\Delta_{i_0})_{s_0}\right)}(\X/S, \Delta_{i_0}).$$ Hence, the latter set is dense in $S$. It is clear that the projective elementary family $(\X/S, \Delta_{i_0})$ and the klt $0$-pair $(X, (\Delta_{i_0})_{s_0})$ with $\mathbb{Q}$-coefficients satisfy the hypotheses of \hyperref[case1]{Step 1}.
	Then, there is an \'etale base change $U \to S$ such that $(\X, \Delta_{i_0})\times_S U \approx (X, (\Delta_{i_0})_{s_0})\times_\Ko U$ over $U$ by \hyperref[case1]{Step 1}. Thus, we can assume that $\X = X\times_\Ko S$, and $\Delta_{i_0} = (\Delta_{i_0})_{s_0}\times_\Ko S$.
		
	Second, we derive a version of Zariski decomposition for the $\mathbb{R}$-divisor $B$. To ease the notation, we write $(\Delta_i)_{s_0} = \Delta_i^0$. By assumption $-K_X\equiv B$; hence $-K_X$ is pseudo-effective and $\kappa(-K_X)\ge 0$. By \cite[Theorem~14.14]{Bad01} there exists a (unique) Zariski decomposition $-K_X=P_K+N$, where $P_K$ is a nef $\mathbb{Q}$-divisor, $N$ is an effective $\mathbb{Q}$-divisor whose components have a negative-definite intersection matrix, and $P_K\cdot N=0$. From the properties of the Zariski decomposition it follows that $N\le \Delta_i^0$ and $P_i=\Delta_i^0-N$ is an effective $\mathbb{Q}$-divisor such that $P_i\equiv P_K$ for all indices $i$. By the semiampleness theorem \cite[Theorem~8.1]{Fuj12} we have $K_X+\Delta_i \sim_\mathbb{Q}0$. Again by semiampleness, for the klt pair $(X,\Delta_i+\epsilon P_i)$ with $0<\epsilon\ll 1$, the effective nef $\mathbb{Q}$-divisor $\epsilon P_i$ is semiample.
		
	Next, we construct a Zariski decomposition for $\B$. For that, we prove that the $\mathbb{R}$-divisor $\B$ is numerically constant over $S$. To be precise, the following lemma holds:
		
		\begin{lemma}
			\label{lemma: numerically_equivalent}
			We have $\B_s \equiv B$ for every closed point $s\in S$.
		\end{lemma}
		
		\begin{proof}
			Write $B=\sum_i b_i B_i$ as a sum of distinct prime components, and set $\Gamma=\{b_i\ne 0\}$ with $n=\#\Gamma$. For $j=1,\dots,n$ define the reduced divisor $\overline{B}_j=\sum_{i:\, b_i=a_j} B_i$, where $a_j\in\Gamma$ and $a_j<a_{j+1}$. Then we write $B=\sum_{j=1}^n a_j \overline{B}_j$ slightly abusing the notation (since $\overline{B}_j$ is not necessarily prime). Similarly, $\B=\sum_{j=1}^n a_j \overline{\B}_j$. For every closed point $s\in S$ we have
			\[
			(\overline{\B}_j)_s-\overline{B}_j \ =\sum_{V\in\supp(\overline{\B}_j)}(V_s-V_{s_0}).
			\]
			Thus the divisors $(\overline{\B}_j)_s$ and $\overline{B}_j$ are algebraically equivalent \cite[Example~10.3.2]{Ful98} since the base is a non-singular variety. As the algebraic equivalence implies the numerical equivalence, $(\overline{\B}_j)_s\equiv \overline{B}_j$. It follows that
			\[
			\B_s=\sum_{j=1}^n a_j(\overline{\B}_j)_s \ \equiv\  \sum_{j=1}^n a_j \overline{B}_j = B.
			\]
		\end{proof}
		
		Therefore, we have $\B_s\equiv B$ for every closed point $s\in S$ according to \cref{lemma: numerically_equivalent}. Similarly, $(\Delta_i)_s \equiv \Delta_i^0$ for every index $i = 1, 2, \dots, I$ and every closed point $s\in S$. From the properties of the Zariski decomposition and the elementarity of the family $(X\times_{\mathbb{C}} S/S,\B)$ it follows that
		\[
		\B=\sum_{i=1}^I r_i\mathcal{P}_i+N\times_\Ko S,
		\]
		where $\mathcal{P}_i$ is an effective $\mathbb{Q}$-divisor such that $(\mathcal{P}_i)_s\sim_\mathbb{Q} P_K$ for every closed point $s\in S$.
		
		Finally, we use the Zariski decomposition for $\B$, finiteness results for automorphisms groups to show that the set $\Iso_{(X, B)}(\X/S, \B)$ is constructible. Let $\overline{B}_i, \overline{\B}_i$ be reduced divisors constructed in the proof of \cref{lemma: numerically_equivalent}, where $i = 1, 2, \dots, n$. Then, for each index $j\in \{1, 2, \dots, n\}$ define a closed subset $ Z_j^{(X, B)} = \bigcup_{i= j}^n \overline{B}_i$ of $Z_0^{(X, B)} = X$ with the structure of a reduced closed subscheme. The finite sequence $\underline{Z}^{(X, B)} = (Z_0^{(X, B)}; Z_1^{(X, B)}, \dots, Z_n^{(X, B)})$ is called the \text{flag scheme associated with the pair} $(X, B)$. The number $l(\underline{Z}^{(X, B)}) = n$ is called \text{the length} of $\underline{Z}^{(X, B)}$. We remark that given two pairs $(X, B)$ and $ (X^\prime, B^\prime)$ with an isomorphism $\varphi: X\to X^\prime$ we have that $\varphi$ is a log isomorphism of pairs $(X, B), \, (X^\prime, B^\prime)$ if and only if $\varphi$ is an isomorphism of the flag schemes $\underline{Z}^{(X, B)}, \, \underline{Z}^{(X^\prime, B^\prime)}$, that is, $l\left(Z^{(X^\prime, B^\prime)}\right) = n$ and  $\varphi\left({Z}^{(X, B)}_i\right) = {Z}^{(X^\prime, B^\prime)}_i$ for each $i = 0, 1, \dots, n$. Similarly, $\underline{\Z}^{(\X, \B)} = (\Z^{(\X, \B)}_0, \Z^{(\X, \B)}_1, \dots, \Z^{(\X, \B)}_n)$ is the flag scheme associated with $(\X, \B)$. Shrinking the base, we can assume that $f|_{\Z^{(\X, \B)}_i}: \Z^{(\X, \B)}_i \to S$ is a flat projective morphism for all~$i$. By \cite[Theorem 4.5.1]{Ser06}, the flag Hilbert functor is represented by a projective scheme $\Hilb$ called the flag Hilbert scheme of $X$ (see \cite[p. 230]{Ser06} for relevant definitions). By universality of $\Hilb$, there is a morphism $z: S\to \Hilb$ that sends a closed point $s$ to $$\left(\underline{\Z}^{(\X, \B)}\right)_s = \left((\Z^{(\X, \B)}_0)_s, (\Z^{(\X, \B)}_1)_s, \dots, (\Z^{(\X, \B)}_n)_s\right).$$
		
		We show that there is a constructible subset $H\subseteq\Hilb(\Ko)$ such that $z^{-1}(H) = \Iso_{(X, B)}(\X/S, \B)$. Let $\mathbf{Aut}(X)$ be the group scheme representing the automorphism functor for $X$ (see \cite[Theorem 5.23]{FGA05}). Analogously,  $\mathbf{Aut}(X, P_{i_0})$ is the group scheme such that $\mathbf{Aut}(X, P_{i_0})(\Ko) = \text{Aut}(X, P_{i_0}) = \{\varphi\in \text{Aut}(X): \varphi_*(P_{i_0}) = P_{i_0}\}$\footnote{$T(\Ko)$ is the underlying set (endowed with the induced Zariski topology) of closed points for a scheme $T$.}. Furthermore, $\mathbf{Aut}(X, P_{i_0})$ is a closed subscheme of $\mathbf{Aut}(X)$. Let us remember that for every $s\in \Iso_{(X, B)}(\X/S, \B)$ there is $\varphi_s\in \text{Aut}(X)$ such that $(\varphi_s)_*(\Delta^0_{i_0}) = \Delta^0_{i_0}$. This implies $(\varphi_s)_*(P_{i_0}) = P_{i_0}$ because the Zariski decomposition is unique. If $\kappa(P_K)=0$, then $P_K\equiv 0$ and $B\equiv N$. But, the inequality $N\le B$ implies $B=N$. By \cref{lemma: numerically_equivalent} we obtain $\B = N\times_\Ko S$. In other words, when $\kappa(P_K) = 0$, the family $(\X/S, \B)$ is trivial, and \cref{th: iso_pairs} follows. If $\kappa(P_K) = 2$, then the $\mathbb{Q}$-divisor $P_{i_0}$ is big and nef. According to \cite[Lemma 2.23]{Zha09}, the group scheme $\mathbf{Aut}(X, P_{i_0})$ is of finite type. Consider the natural morphism $\mathbf{Aut}(X)\times_\Ko \Hilb \to \Hilb$ defined on closed points by $(\varphi, (Z_0, Z_1, \dots, Z_n)) = (\varphi(Z_0), \varphi(Z_1), \dots, \varphi(Z_n))$. By Chevalley's Theorem \cite[Theorem 1.8.4]{EGAIV}, the image $H$ of the composition $$\text{Aut}(X, P_{i_0})\times_\Ko \{\underline{Z}^{(X, B)}\} \hookrightarrow \text{Aut}(X)\times_\Ko \Hilb(\Ko) \to \Hilb(\Ko)$$ is constructible because  $\mathbf{Aut}(X, P_{i_0})$ is of finite type. By construction, $z^{-1}(H) = \Iso_{(X, B)}(\X/S, \B)$ is a constructible dense subset of $S(\Ko)$. Hence, \cref{th: iso_pairs} follows from the standard isomorphism functor argument.
		\begin{lemma}[{cf. \cite[p. 393]{Amb05}}]
			\label{lem: constr_implies_etale}
			Suppose $\textnormal{Iso}_{(X, B)}(\X/S, \B)$ is constructible and dense in $S$. Then there is an \'etale base change $U\to S$ such that $(\X, \B)\times_S U \approx (X, B)\times_\Ko U$ over~$U$.
		\end{lemma}
		\begin{proof}
			Let $\mathcal{I}so_S((X, B)\times_\Ko S/S, (\X, \B)/S)$ be the standard isomorphism functor from the category of locally Noetherian schemes to the category of sets. It is represented by a scheme $\textbf{Iso}_S$ together with a morphism $I: \textbf{Iso}_S\to S$ locally of finite type (cf. \cite[Theorem 5.23]{FGA05}). Note that constructibility and density of $\textnormal{Iso}_{(X, B)}(\X/S, \B)$ implies it includes an open dense subset of $S$. Hence we can assume that the morphism $I: \textbf{Iso}_S\to S$ is surjective after shrinking the base. By the generic smoothness on the source in characteristic $0$, there is an open non-empty subset $\widetilde{U}\subseteq \textbf{Iso}_S$ such that $I: \widetilde{U}\to S$ is smooth. Then $I$ admits a section \'etale locally, as required.
		\end{proof}
		It remains to consider the case $\kappa(P_K) = 1$. Let $p: X\to Y = \textbf{Proj}\, R(X, P_{i_0})$ be the projective fibration onto the canonical model for $P_{i_0}$, where $R(X, P_{i_0}) = \bigoplus_{m\ge 0} H^0\left(X, \mathcal{O}_X\left(\lfloor mP_{i_0}\rfloor\right)\right)$. By construction, $Y$ is normal, and $p_*\mathcal{O}_X\cong \mathcal{O}_Y$ (see \cite[Proposition 7.6]{Deb01}). Let $\mathbf{Aut}(p)\subseteq \mathbf{Aut}(X)\times_\Ko\mathbf{Aut}(Y)$ denote the group scheme such that $$\mathbf{Aut}(p)(\Ko) = \{(\varphi, \psi)\in \text{Aut}(X)\times\text{Aut}(Y): \psi\circ p = p\circ \varphi \}.$$
		\begin{lemma}
			The natural projection $\mathbf{Aut}(p)\to \mathbf{Aut}(X)$ is an isomorphism
		\end{lemma} 
		\begin{proof}
			Indeed, for every (locally Noetherian) scheme $T$ over $\Ko$ and every automorphism $\varphi\in \text{Aut}(X\times_\Ko T/T)$ we get $$\varphi_*(P_{i_0}\times_\Ko T + N\times_\Ko T)=\varphi_*(\Delta_{i_0}^0)\equiv_T -K_{X\times_{\Ko}T/T}\equiv_T \Delta_{i_0}^0 =  P_{i_0}\times_\Ko T + N\times_\Ko T.$$ Note that $\left((\varphi_t)_*(\Delta_{i_0}^0) - N\right) + N$ is the Zariski decomposition of $(\varphi_t)_*(\Delta_{i_0}^0)$ for any closed point $t\in T$. From uniqueness of the Zariski decomposition it follows that $\varphi_*(N\times_{\Ko} T) = N\times_{\Ko} T$ , and $\varphi_*(P_{i_0}\times_\Ko T)\equiv_T P_{i_0}\times_\Ko T$. Hence, every fiber of $p\times_\Ko id_T$ is contracted by $(p\times_\Ko id_T)\circ\varphi$, and vice versa. Therefore, there exists an automorphism $\psi: Y\times_\Ko T \to Y\times_\Ko T$ over $T$ by the rigidity lemma \cite[Lemma 1.15]{Deb01} provided that $(p\times_{\Ko}id_T)_*\mathcal{O}_{X\times_{\Ko} T}\cong \mathcal{O}_{Y\times_{\Ko} T}$.
		\end{proof}
		\begin{lemma}
			\label{lemma: aut_relative}
			We have $\varphi_*(B) = B$ for all  $\varphi\in\textnormal{Aut}(X/Y)$.
		\end{lemma}
		\begin{proof}
			 Recall that $\Delta_i^0 \sim_\mathbb{Q} -K_X$ for all $i=1, 2, \dots, I$. Since $\varphi_*(\Delta_i^0) \sim_\mathbb{Q} \Delta^0_i$ we obtain $\varphi_*(P_i) + \varphi_*(N)\sim_\mathbb{Q} P_i + N$. From properties of the Zariski decomposition it follows that $\varphi_*(N) = N$. Hence, $\varphi_*(P_i)\sim_{\mathbb{Q}} P_i\sim_{\mathbb{Q}} P_{i_0}$ for all $i$. Therefore, $\supp P_i = p^{-1}(\{p_{ij}\})$ for some closed points $p_{ij}\in Y$. This implies that $\varphi_*(B) = B$ for all  $\varphi\in\text{Aut}(X/Y)$.
		\end{proof}
		 Note there is a natural exact sequence of group schemes:
		\begin{equation*}
			1\to \textbf{Aut}(X/Y) \to \textbf{Aut}(X) \cong \textbf{Aut}(p)\to \textbf{Aut}(Y).
		\end{equation*}
		By \cite[Theorem 5.39]{Mil17}, the quotient $\mathbf{A} = \textbf{Aut}(X) / \textbf{Aut}(X/Y)$ is a group scheme isomorphic to a closed subscheme of $\textbf{Aut}(Y)$. It is well-known that $\textbf{Aut}(Y)$ is of finite type \cite[Chapter IV]{Har77}. Hence $\mathbf{A}$ is of finite type as well. Then the composition $$\text{Aut}(X)\times_\Ko \{\underline{Z}^{(X, B)}\} \hookrightarrow \text{Aut}(X)\times_\Ko \Hilb(\Ko) \to \Hilb(\Ko)$$ factors through the natural map $\textbf{A}(\Ko)\times_\Ko \{\underline{Z}^{(X, B)}\} \to \Hilb(\Ko)$ due to \cref{lemma: aut_relative}. Again by Chevalley's Theorem, the image $H^\prime$ of the above composition is constructible. Hence $z^{-1}(H) = \Iso_{(X, B)}(\X/S, \B)$ is a constructible as well. Thus, \cref{lem: constr_implies_etale} concludes the proof of \cref{th: iso_pairs}.
	\end{proof}

	\section{Main results}
	\label{sec: main}
	In this section we prove \cref{main_intro}.
	
	\begin{proposition}
		\label{lemma: H-boundedness}
		Let $\Cc$ be a class of projective pairs $(X_\alpha,B_\alpha)$ of fixed dimension $d\in \mathbb{Z}_{>0}$. If $\Cc$ has bounded polarization, then $\Cc$ is a subclass of a bounded class of projective pairs.
	\end{proposition}
	
	\begin{proof}
		\textbf{Step 1.}\label{step1} We first prove the proposition in the case when the $\mathbb{R}$-divisors $B_\alpha$ are reduced, that is, $B_\alpha=\sum_i B_{i\alpha}$. We also assume that the number of prime components is the same for all $B_\alpha$ and equals some $n\in\mathbb{Z}_{\ge0}$. Note that the expression $B_\alpha=\sum_{i=1}^n b_{i\alpha}B_{i\alpha}$ imposes an ordering on the prime divisors $B_{i\alpha}$. By the standard Hilbert--Chow schemes argument (see the proof of \cref{lemma: H-boundedness(varieties)} or the proof of \cite[Lemma 2.20]{Bir19}) there exist flat projective morphisms $\X'\to S'$ and $\B_i'\to T_i$ together with closed embeddings $\varphi_\alpha:X_\alpha\hookrightarrow \Pp^M_{\mathbb{C}}$ such that $\varphi_\alpha(X_\alpha)=\X'_{s_\alpha}$ and $\varphi_\alpha(B_{i\alpha})=(\B_i')_{t_{i\alpha}}$ for some closed points $s_\alpha\in S'$ and $t_{i\alpha}\in T_i$.
		Set $S=S'\times_{\mathbb{C}}\prod_i T_i$ and $\X=\X'\times_{\mathbb{C}}\prod_i T_i$. Define
		\[
		\B_i=\left(\B_i'\times_{\mathbb{C}} S'\times_{\mathbb{C}}\Big(\prod_{j\ne i} T_j\Big)\right)\bigcap \X \ \subseteq\ \Pp^M_{S}
		\]
		to be the scheme-theoretic intersection endowed with the structure of a reduced closed subscheme. By \cref{prop: open_properties} the set
		$
		\{\,s\in S:\ \X_s \ \text{is geometrically integral and geometrically normal} \ \}
		$
		is open. Therefore, after shrinking the base, all fibers of $\X\to S$ are normal projective varieties of dimension $d$. After a base change we may assume that each closed subscheme $\B_i$ is flat over $S$. Since the set
		$
		\{\,s\in S:\ (\B_i)_s \ \text{is a geometrically integral scheme}\ \}
		$
		is open according to \cref{prop: open_properties}, after shrinking the base we may assume that $(\B_i)_s$ is a prime divisor on $\X_s$ for all $s\in S$. After shrinking the base further we may assume $(\B_i)_s\ne (\B_j)_s$ for all $s\in S$ whenever $i\ne j$. Put $\B = \sum_i \B_i$. Then $(\X/S,\B)$ is an elementary projective family, as required.
		
		\textbf{Step 2.} Now, we prove the proposition in general. By the definition of a class with bounded polarization there exists a finite set $\Gamma=\{a_1,\dots,a_r\}\subset \mathbb{R}$ such that $B_\alpha\in\Gamma$, and the number of prime components of the divisors $B_\alpha$ is bounded by some $n\in\mathbb{Z}_{>0}$ (see \cref{remark: components}). Then $\Cc$ splits into finitely many subclasses $\Cc_{n_1, n_2, \dots,n_r}$ corresponding to the pairs $(X_\alpha,B_\alpha)$ such that
		\[
		n_j=\#\{\,B_{i\alpha}\in \supp(B_\alpha):\ b_{i\alpha}=a_j\,\},\qquad 0\le\sum_{j=1}^r n_j\le n.
		\]
		Hence it suffices to prove the proposition for one fixed class $\Cc_{n_1, n_2, \dots,n_r}$. For each index $j=1, \dots ,r$ and every pair $(X_\alpha, B_\alpha)$ in  $\Cc_{n_1, n_2, \dots,n_r}$ define the reduced divisor $\ B_\alpha^{(j)}=\textstyle\sum_{i: \, b_{i\alpha}=a_j} B_{i\alpha}$. Let $\Cc_j$ be the class consisting of all pairs $\bigl(X_\alpha,\ B_\alpha^{(j)}\bigr)$. According to \hyperref[step1]{Step 1}, the classes $\Cc_j$ of pairs are subclasses of bounded classes of projective pairs. It follows that for each index $j=1,\dots r$ there exists a projective elementary family $(\X^{(j)}/S^{(j)},\B^{(j)})$ with $\B^{(j)} = \sum_{i=1}^{n_j} \B_i^{(j)}$ such that every pair in $\Cc_j$ is isomorphic to a fiber of $(\X^{(j)}/S^{(j)},\B^{(j)})$, and $\X^{(j)}/S^{(j)}\subseteq \Pp^M_{S^{(j)}}$. We use the convention that if $n_j=0$ then $\B^{(j)} = 0$. Define
		\[
		S=\prod_{j=1}^r S^{(j)}, \qquad
		\X=\bigcap_{j=1}^r\Bigl(\X^{(j)}\times_\Ko \prod_{k\ne j} S^{(k)}\Bigr) \subseteq \Pp^M_{S},\qquad
		\B=\sum_{j=1}^r a_j \sum_{i=1}^{n_j}\Bigl(\bigl(\B_i^{(j)}\times_\Ko \prod_{k\ne j} S^{(k)}\bigr)\bigcap \X\Bigr).
		\]
		By the same argument as in the first step, we conclude that, after shrinking and stratification the base, the family $(\X/S,\B)$ satisfies the conditions of \cref{def: bounded_pairs}.
	\end{proof}
	
	\begin{proposition}
		\label{prop: klt_fibers}
		Let $(\X/S,\B)$ be an elementary projective family. Suppose $(\X,\B)$ is a klt (resp. trm) log pair. Then there is an open dense subset $U\subseteq S$ such that the log pair $(\X_s,\B_s)$ has klt (resp. trm) singularities for every closed point $s\in U$.
	\end{proposition}
	
	\begin{proof}	
		We treat the case when $(\X,\B)$ is klt, as the terminal case is analogous. Let $q:\X'\to \X$ be a log resolution and write
		\[
		K_{\X'}+\B'=q^*(K_\X+\B).
		\]
		We may assume that $(\X'/S,\B')$ is a projective elementary family. After shrinking the base, $(\X',\B')$ is a log smooth pair over $S$ by generic smoothness. Hence for every closed point $s\in S$ the induced morphism $q_s: (\X'_s,\B'_s)\to (\X_s,\B_s)$ is a log resolution. As $(\X,\B)$ is a klt log pair, all coefficients (in the decomposition into prime components) of the $\R$-divisor $\B'$ are less than $1$. Since $(\X'/S,\B')$ is an elementary family, the coefficients of $\B'_s$ are also less than $1$. In other words, $(\X_s,\B_s)$ is klt for all closed $s\in S$.
	\end{proof}
	
	\begin{proposition}[{cf. \cite[Theorem 1]{Kaw08}}]
		\label{prop: unique_model}
		Let $(X, B)$ and $(X', B')$ be projective wlc trm models of dimension $2$. Assume there is a birational map $\varphi: X \dashrightarrow X'$ such that $\varphi_*(B) = B'$. Then $(X, B)$ and $(X', B')$ are crepant birationally equivalent. Moreover, they are log isomorphic.
	\end{proposition}
	
	\begin{proof}
		We follow the proof in \textit{loc. cit.} Let $Y$ be a common log resolution for $(X,B)$ and $(X',B')$, and let $q:Y\to X, q': Y\to X'$ be the corresponding morphisms. Let $\{E_i\}$ be the union of all exceptional divisors for $q$ and $q'$. Then we write:
		\begin{equation*}
			K_Y = q^*(K_{X} + B) - (q)_*^{-1}B + E = q'^*(K_{X'}+ B') - (q')_*^{-1}B' + E', 
		\end{equation*}
		where $E = \sum_i a_i E_i$ and $E^\prime  = \sum_i b_i E_i$. Since $(X, B)$ and $(X', B')$ are terminal, we have $E$ and $E^\prime$ are effective, $\supp(E) = \E(q)$ and $\supp(E') = \E(q')$.  We show that $E = E'$. We set $F = \sum_i \min\{a_i, b_i\}E_i$, $\overline{E} = E - F, \ \overline{E^\prime} = E^\prime - F$. In addition, $(q^\prime_*)^{-1}B^\prime \le (q_*)^{-1} B$ because $\varphi$ may contract some irreducible components of $B$.  Note that $\supp(\overline{E})\cap \supp(\overline{E^\prime}) = \emptyset$. Suppose $\overline{E}\neq 0$. Then there a curve $C\in \supp \overline{E}$ such that $\overline{E}\cdot C < 0$  according to the negativity lemma \cite[Lemma 3.6.2]{BCHM10} applied to $\overline{E}$ and $q$. Let us remember that $K_X + B$ and $K_{X^\prime} + B^\prime$ are nef. Hence we obtain the following contradiction:
		\begin{equation*}
			0\le \left((q^\prime)^*(K_{X^\prime} + B^\prime) + \left(q_*^{-1}B - (q_*^\prime)^{-1}B^\prime\right) + \overline{E^\prime}\right) \cdot C = \left(q^*(K_X + B) + \overline{E}\right) < 0.
		\end{equation*}
		Therefore, $\overline{E} = 0$ and $E\le E'$. Similarly, the reverse inequality holds as well. Indeed, by the negativity lemma applied to $\left(q_*^{-1}B - (q_*^\prime)^{-1}B^\prime\right) +\overline{E^\prime}$ and $q^\prime$, there a curve $C\in \supp (\overline{E^\prime})$ such that $\left(\left(q_*^{-1}B - (q_*^\prime)^{-1}B^\prime\right) +\overline{E^\prime}\right)\cdot C < 0$ if $\overline{E'}\neq 0$. This implies the following contradiction:
		\begin{equation*}
			0>  \left((q^\prime)^*(K_{X^\prime} + B^\prime) + \left(q_*^{-1}B - (q_*^\prime)^{-1}B^\prime\right) + \overline{E^\prime}\right) \cdot C = \left(q^*(K_X + B) + \overline{E}\right)\cdot C \ge 0.
		\end{equation*}
		Consequently, we get $E = E'$, that is, the pairs $(X, B)$ and $(X', B')$ are crepant birationally equivalent. In addition, the equality $\E(q) = \E(q')$ implies that $\varphi: X\dashrightarrow X'$ is an isomorphism in codimension 1. Then the pairs $(X, B)$ and $(X', B')$ are log isomorphic provided $\dim X = \dim X' = 2$.
	\end{proof}
	
	 Let $\Cc$ be a class of pairs. By $|\Cc|$ we denote the set of (log) isomorphism classes in $\Cc$.
	\begin{theorem}
		\label{main}
		Suppose $(X, B)$ is a projective trm 0-pair of dimension 2. Let $\Cc$ be the class of all projective wlc klt models $(X_\alpha,B_\alpha)$ in the $0$-class of $(X,B)$. Consider the class $\Dd$ of all varieties $X_\alpha$ corresponding to the pairs $(X_\alpha,B_\alpha)$ in~$\Cc$. If the class $\Dd$ has bounded polarization, then the set $|\Cc|$ is finite.
	\end{theorem}
	
	\begin{lemma}
		\label{lemma: deg_boundary}
		The class $\Cc$ has bounded polarization, and all $B_\al\in[0,1]$.
	\end{lemma}
	\begin{proof}
		We first show that there exists a finite set $\Gamma\subset[0,1]$ such that $B_\al\in \Gamma$. By \cite[Corollary 1.4.3]{BCHM10}, for every $(X_\al, B_\al)$ in $\Cc$ there is a terminal log pair $(X_\al^\text{trm}, B_\al^\text{trm})$ with a birational projective morphism $q_\al: X_\al^\text{trm}\to X_\al$ such that $K_{X_\al^\text{trm}} + B_\al^\text{trm} = q_\al^*(X_\al + B_\al)$, and $B_\al^\text{trm}$ is a boundary. Note that $(X_\al^\text{trm}, B_\al^\text{trm})$ is a trm wlc model in the 0-class of $(X, B)$. According to \cref{prop: unique_model}, we can assume that $(X_\al^\text{trm}, B_\al^\text{trm}) = (X, B)$ for all $(X_\al, B_\al)$ in $\Cc$. Set $\Gamma = \{b_i: i=1, 2, \dots, n\}\subset [0,1]$, here $B = \sum_{i=1}^n b_i B_i$. Then $B_\alpha \in \Gamma$ because $B_\al = (q_\al)_*B$. Moreover, $\# \supp(B_\alpha) \le \# \supp(B) \le~n$.
		
		Next we prove that there exists $N\in\mathbb{Z}_{>0}$ and very ample Cartier divisors $H_\alpha$ on $X_\alpha$ such that $B_\alpha\cdot H_\alpha\le N$ for all log pairs $(X_\alpha,B_\alpha)$ in $\Cc$. By \cref{lemma: H-boundedness(varieties)} there exist finitely many projective families $\X^{(j)}/S^{(j)}$ such that all $X_\alpha$ are isomorphic to fibers of the families $\X^{(j)}/S^{(j)}$. Without loss of generality, we may assume there is a single family $\X/S$ and $X_\alpha=\X_{s_\alpha}$ for some closed points $s_\alpha\in S$. Choose a very ample invertible sheaf $\Li$ on $\X/S$ and set $H_\alpha=\Li_{s_\alpha}$. By elementarity of $\X/S$ there exists a polynomial $p(m)\in\mathbb{Q}[m]$ such that the Euler characteristic $\chi(mH_\alpha)$ equals $p(m)$ for all $s_\alpha\in S$. As klt singularities are rational, we have $\chi(m q_\alpha^*H_\alpha)=\chi(mH_\alpha)$, where $q_\alpha:Y_\alpha\to X_\alpha$ is a resolution of singularities. Using the Riemann–Roch formula on the smooth surface $Y_\alpha$ and the projection formula, we obtain
		\[
		\chi(mH_\alpha)=\frac{H_\alpha^2}{2}m^2-\frac{K_{X_\alpha}\!\cdot H_\alpha}{2}m+\chi(\mathcal{O}_{Y_\alpha}).
		\]
		Hence the intersection numbers $B_\alpha\cdot H_\alpha=-K_{X_\alpha}\cdot H_\alpha$ are the same for all $X_\alpha$ and equals $N=2p^\prime(0)$.
		
		It remains to bound $(B_\alpha)_{\mathrm{red}}\cdot H_\alpha$. Let $b_{\text{min}}=\min(\Gamma\setminus\{0\})$, and write $B_\alpha=\sum_i b_{i\alpha}B_{i\alpha}$ with $b_{i\alpha}\neq 0$. Then
		$$
		b_{\text{min}}\,B_{i\alpha}\cdot H_\alpha\ \le\ \sum_i b_{\text{min}}\,B_{i\alpha}\cdot H_\alpha\ \le B_\alpha \cdot H_\alpha \le\ N.
		$$
		Recall that $\#\supp (B_\alpha) \le n$. Therefore, $(B_\alpha)_{\mathrm{red}}\cdot H_\alpha \le n\,\dfrac{N}{b_{\text{min}}}$.
	\end{proof}
	
	\begin{proof}[Proof of \cref{main}]
		By \cref{lemma: deg_boundary} and \cref{lemma: H-boundedness}, we have finitely many elementary projective families $(\X^{(j)}/S^{(j)},\B^{(j)})$ such that all log pairs in $\Cc$ are isomorphic to some fibers of these families. Without loss of generality, we assume there is a single elementary projective family $(\X/S,\B)$. Let $\{s_\alpha:\alpha\in I\}\subseteq S$ be the set of closed points parametrizing all pairs in $\Cc$. Replacing the base by the closure of $\{s_\alpha\}_{\alpha\in I}$, we may assume that $\{s_\alpha\}_{\alpha\in I}$ is dense in $S$. Let $\mathcal{F}$ be the set of all closed subsets $S^\prime \subseteq S$ such that $\{s_\alpha\}_{\alpha\in I}\cap S^\prime$ is dense in $S^\prime$. Note that $\mathcal{F}$ is a well-founded set with respect to inclusion as $S$ is Noetherian. Then the set of all minimal elements for $\mathcal{F}$ is $\{\{s_\alpha\}: \al \in I\}$. We say that a property $P(S^\prime)$ holds for $S^\prime \in \mathcal{F}$ if the set $|\{(\X_s, \B_s): s\in S^\prime, s \text{ is closed}\}|$ is finite. We shall proceed by Noetherian induction on $\mathcal{F}$. Clearly, $P(\{s_\al\})$ holds for all $\al\in I$ (Noetherian induction base). Let $S^\prime\in \mathcal{F}$ be a non-minimal element of $\mathcal{F}$. Assume that $P(S^{\prime\prime})$ is true for all proper subsets $S^{\prime\prime}\subset S^\prime$ (Noetherian induction step). We claim that $P(S^\prime)$ holds. To prove the claim, it suffices to show that there exists a log pair $(X^\prime, B^\prime)$ and a dense open subset $U\subseteq S^\prime$ such that $(\X_s,\B_s)\approx (X^\prime,B^\prime)$ for all closed points $s\in U$. Then the finiteness of log pairs in the family follows by Noetherian induction (cf. \cite[Chapter II, Exercise 3.16]{Har77}).
		
		Without loss of generality, we may assume that $((\X\times_S S^\prime)/S^\prime, \B\times_S S^\prime)$ is a projective elementary family. To simplify the notation, we write $(\X/S, \B)$ for the corresponding new family. By \cite[Proposition~2.4]{HX15}\footnote{ Although, the proposition in \textit{loc. cit.} is stated for $\mathbb{Q}$-divisor $\B$, their proof is valid for $\R$-divisor $\B$.}, after shrinking the base we may assume that $(\X,\B)$ is a klt log pair. Let $\pi:(\X',\B')\to (\X,\B)$ be a $\mathbb{Q}$-factorial terminalization \cite[Corollary~1.4.3]{BCHM10}, so that
		\[
		K_{\X'}+\B'=\pi^*(K_{\X}+\B).
		\]
		Here $(\X',\B')$ has only terminal singularities and $\B'$ is a boundary. We may assume $(\X'/S,\B')$ is an elementary projective family. By \cref{prop: klt_fibers}, we can assume that for all $s\in S$ the log pair $(\X'_s,\B'_s)$ has terminal singularities. Moreover, each log pair $(X_\alpha,B_\alpha)$ in $\Cc$ is isomorphic to $(\X_{s_\alpha},\B_{s_\alpha})$, and its terminalization is isomorphic to $(\X'_{s_\alpha},\B'_{s_\alpha})$. By \cref{prop: unique_model}, all log pairs $(\X'_{s_\al},\B'_{s_\al})$ are log isomorphic to the fixed log pair $(X,B)$. Let us remember that $\X^\prime/S$ is smooth, $\X^\prime$ is non-singular, and $(K_{\X'} + \B')|_{\X'_s} = K_{\X'_s} + \B'_s$ for all closed points $s\in S$ (see \cref{sec: isotriviality}).
		
		Next, we prove that $K_{\X'} + \B' \sim_{S, \R} 0$ after shrinking the base. To be definite, we can assume that $(X, B) = (\X_{s_0}^\prime, \B_{s_0}^\prime)$ for a general point $s_0\in S$. According to \cite[Proposition 1.4.14]{Laz04} the sets $U^+ = \{s\in S: (K_{\X'} + \B')|_{\X'_s} \text{ is nef}\}$ and $U^- = \{s\in S: -(K_{\X'} + \B')|_{\X'_s} \text{ is nef}\}$ are (at most) countable intersection of open subsets of $S$. By construction, $\emptyset \neq \Iso_{(X, B)}(\X'/S,\B')\subseteq U^+\cap U^-$. Hence, $U^+\cap U^-$ contains the generic point $\eta\in S$, i. e. $K_{\X'_\eta} + \B'_\eta \equiv 0$. By $K$ we denote an algebraic closure for the residue field $\kappa(\eta)$ of $\eta\in S$. Let $\X'_K, \B'_K$ denote the corresponding pullbacks. Then $K_{\X'_K} + \B'_K \equiv 0$. By  the Lefschetz principle, we can apply \cite[Theorem 8.1]{Fuj12} to $(\X'_K, \B'_K)$. Therefore, the $\R$-divisor $\B'_K$ can be written as an $\mathbb{R}$-linear combination of effective $\mathbb{Q}$-divisors $(\Delta_1)_K, (\Delta_2)_K, \dots, (\Delta_I)_K$ with coefficients $r_1, r_2, \dots, r_I\in \mathbb{R}$ respectively such that $K_{\X'_K} + (\Delta_i)_K \sim_{\mathbb{Q}} 0$ for all $i = 1, 2, \dots, I$, and $\sum_{i=1}^I r_i = 1$. Note that all $(\Delta_i)_K$ are defined over $\kappa(\eta)$ because $\cup_i\text{supp}\left((\Delta_i)_K\right) = \supp (\B'_K)$. Hence we can write $(\Delta_i)_K = \Delta_i\times_{\kappa(\eta)} K$ for some $\mathbb{Q}$-divisors on $\X'_\eta$. Choose an integer $m\gg 1$ such that $m(\Delta_i)_K$ is Cartier, and $m(K_{\X'_K} + (\Delta_i)_K) \sim 0$ for all $i$. Put $D_i = m\left(K_{\X'_\eta} + \Delta_i\right)$, and $(D_i)_K = D_i\times_{\kappa(\eta)} K$. By the flat base change theorem \cite[Chapter III, Proposition 9.3]{Har77}, we have $H^0\left(\X'_{\eta}, \mathcal{O}_{\X'_{\eta}}\left(D_i\right)\right)\otimes_{\kappa(\eta)}K\cong H^0\left(\X'_K, \mathcal{O}_{\X'_K}\left((D_i)_K\right)\right)\approx K$ for all $i = 1, 2, \dots, I$. This implies that $m(K_{\X'_\eta} + \Delta_i)  = \text{div}_{\X'_\eta}(\varphi_i)$ for some rational functions $\varphi_i\in k(\X'_\eta)$. Therefore, $K_{\X_\eta} + \B_\eta = r_1^\prime \text{div}_{\X'_\eta}(\varphi_1) + r_2^\prime \text{div}_{\X'_\eta}(\varphi_2) + \dots + r_I^\prime\text{div}_{\X'_\eta}(\varphi_I)$, where $r_i^\prime = r_i/m$. Note that the functions $\varphi_i\in k(\X'_\eta)$ can be viewed as rational functions on $\X'$. Then the $\R$-divisor $E = K_{\X'} + \B' - r_1^\prime \text{div}_{\X'}(\varphi_1) - r_2^\prime \text{div}_{\X'}(\varphi_2) - \dots - r_I^\prime\text{div}_{\X'}(\varphi_I)$ is disjoint from $\X'_\eta$, that is, $E$ is vertical over $S$. Furthermore, $E = f^*(L)$ for some $\mathbb{R}$-divisor on $S$ because all fibers of $\X'/S$ are reduced and irreducible. Hence, $K_{\X'} + \B \sim_\R f^*(L)$, as required.
		
		By \cref{th: iso_pairs}, after shrinking the base and passing to an étale cover we may suppose that $\X'=X\times_{\mathbb{C}} S$  $\B^\prime = B\times_\Ko S$. Recall that $\pi: (X, B)\times_\Ko S \to (\X, \B)$ is the terminalization morphism. We will show that the exceptional set $\E(\pi)$ contains only prime divisors.
	\begin{lemma}
		\label{prop: contr_curve}
		Let $\varphi: X\times_\Ko S \to \mathcal{Y}$ be a birational projective contraction over $S$ that contracts a curve $C\subset X$ in the fiber $X\times_\Ko s_0$. Then $\varphi$ contracts $C\times_\Ko S$.
	\end{lemma}
	\begin{proof}
		By assumption, the morphism $\varphi$ contracts the class of the curve $[C\times_{\Ko} s_0]\in N_1(X\times_\Ko S/S)$, but $[C\times_{\Ko} s_0] = [C\times_\Ko s]$ for every closed point $s\in S$. Let $\mathcal{H}$ be an ample Cartier divisor on $\Y/S$. Hence $(\varphi^*\mathcal{H})_s\cdot C = (\varphi^*\mathcal{H})_{s_0}\cdot C = 0$ for every closed point $s\in S$. Then $\varphi$ contracts $C\times_\Ko S$.
	\end{proof}
	Hence $\E (\pi) = \bigcup_{i=1}^n \{E_i\times S\}$ for some prime divisors $E_i$ on $X$. By \cite[Lemma~4.3]{Tot10} there exist positive real coefficients $a_i$ such that $\left(\sum_{i=1}^n a_i E_i\right)\cdot E_j = -1$ for all $j$. Set  $E = \sum_{i=1}^n a_i E_i$. Next, we run the MMP for $(K_X + B + \epsilon E)\times_\Ko S$ over $S$. Let us remark that $K_X\times_\Ko S = K_{X\times_\Ko S/S}$. We prove that an extremal contraction preserves the triviality of the projective elementary family (see \cref{prop: extremal}). First, let us verify the following lemma, which is essentially a version of the rigidity lemma \cite[Lemma 1.15]{Deb01}.
	
	\begin{lemma}
		\label{prop: exc_1=exc_2}
		Suppose that $(X\times_\Ko S/S, 0)$, $(Y\times_\Ko S/S,0)$, and $(\mathcal{Y}/S,0)$ are projective elementary families, and $Y$ is a $\mathbb{Q}$-factorial surface. Let $\varphi: X\times_\Ko S \to \mathcal{Y}$ and $p\times_\Ko id_S: X\times_\Ko S \to Y\times_\Ko S$ be birational projective contractions over $S$ such that $\E(p\times_\Ko id_S)\subseteq \E(\varphi)$. Then the natural map $\varphi\circ(p\times_\Ko id_S)^{-1}: Y\times_\Ko S \dashrightarrow \mathcal{Y}$ is a morphism over~$S$.
	\end{lemma}
	
	\begin{proof}
		Let $\mathcal{A}$ be an ample Cartier divisor on $\mathcal{Y}/S$. Set $\mathcal{D} = (p\times_\Ko id_S)_*(\varphi)^*\mathcal{A}$.  Replacing $\mathcal{A}$ by its sufficiently high multiple, we can assume that $\mathcal{D}$ is a Cartier divisor on $Y\times_\Ko S/S$ because $Y$ is $\mathbb{Q}$-factorial. Note that $(p\times_\Ko id_S)^*(\mathcal{D})\cdot C = 0$ if and only if $\varphi^*(\mathcal{A})\cdot C = 0$ for curves $C\subset X\times_\Ko S$ because $\E(\varphi)=\E(p\times_\Ko id_S)$, and $\dim Y = 2$. Therefore, the negativity lemma \cite[Lemma 3.6.2]{BCHM10} implies $(p\times_\Ko id_S)^*\mathcal{D} = \varphi^*(\mathcal{A})$. In addition, the map $\alpha = (p\times_\Ko id_S)\circ\varphi^{-1}$ is a morphism. Indeed, suppose there is a curve $C\subset X\times_{\Ko} S$ such that $\varphi_*(C)$ is a curve, but $(p\times_\Ko id_S)$ is a point. Then
		$$
		0 = (p\times_\Ko id_S)^*(\mathcal{D})\cdot C = \varphi^*(\mathcal{A})\cdot C = \mathcal{A}\cdot \varphi_*(C) > 0.
		$$
		This is a contradiction. Hence $\alpha$ is a morphism.
	\end{proof}

	\begin{proposition}
	\label{prop: extremal}
	Let $\varphi: X\times_\Ko S \to \mathcal{Y}$ be an extremal contraction for  $(K_X + B + \epsilon E)\times_\Ko S$ over~$S$. Then $\E(\varphi)\subseteq \supp (E \times_\Ko S)$. Moreover, there exists a log pair $(Y, D)$ with (at worst) klt singularities such that
	$
	(\Y, \varphi_* \left(\left(B + \epsilon E \right)\times_\Ko S\right)) \approx (Y, D)\times_\Ko S
	$ over $S$.
	\end{proposition}
	
	\begin{proof}
	By \cref{prop: contr_curve} the contraction $\varphi$ is divisorial. By extremality of $\varphi$, $\E(\varphi) = \{C\times_\Ko S\}$ for some curve $C\subset X$. Since $(K_X + B + \epsilon E)\times_\Ko S \equiv_S \epsilon E\times_{\Ko} S$ and $\epsilon (E\times_{\Ko} S)\cdot C < 0$, we have $C\times S\in \supp(E\times_{\Ko} S)$. Hence $\E(\varphi)\subseteq \supp (E \times_\Ko S)$. By the cone theorem \cite[Theorem 3.7]{KM98} for $K_{X} + B + \epsilon E$, there exists an extremal contraction $p: X \to Y$ of the curve $C$ such that $(Y, D)$ is klt, where $D = (p)_*\left(B + \epsilon E\right)$. Put $p\times_\Ko id_S: X\times_\Ko S \to Y\times_\Ko S$. It is clear that $\text{Exc}(p\times_\Ko id_S) = \{C\times_\Ko S\}$.
	
	By \cref{prop: exc_1=exc_2}, the birational maps $\alpha = \varphi \circ (p\times_\Ko id_S)^{-1}$ and $\alpha^{-1} = (p\times_\Ko id_S)\circ \varphi^{-1}$ are morphisms over $S$, i.e. $\alpha$ is an isomorphism over $S$. Then it induces a log isomorphism
	$
	(Y, D)\times_\Ko S \approx (\Y, \varphi_* \left(\left(B + \epsilon E \right)\times_\Ko S\right))
	$
	because $\text{Exc}(p\times_\Ko id_S) = \text{Exc}(\varphi)$.
	Therefore, \cref{prop: extremal} is proved.
	\end{proof}

	\begin{remark}
	More generally \cite[Theorem~30]{Kol25}, Koll\'ar established that even small modifications preserve the triviality of the family.
	\end{remark}
	
	Recall that $\text{Exc}(\pi) = \supp (E\times_\Ko S)$, where $\pi: (X, B)\times_\Ko S\to (\X,\B)$ is a terminalization. By \cref{prop: extremal}, we may assume that the first step of the MMP for $(K_X + B + \epsilon E)\times_\Ko S$ over $S$ is the projective contraction $p\times_\Ko id_S: X\times_\Ko S \to Y_1\times_\Ko S$ contracting $E_i\times_\Ko S\in \supp (E\times_\Ko S)$ for some $i=1, 2, \dots, n$. Without loss of generality, let $i=1$. Moreover, $(Y_1, D_1)$ is klt, where $D_1 = p_*(B + \epsilon E)$. By \cref{prop: exc_1=exc_2}, there is a morphism $\pi_1: Y_1\times_\Ko S\to \X$ such that $\pi = \pi_1 \circ (p\times_\Ko id_S)$. Hence, for every closed points $s\in S$ we have $(p_* (E\times_{\Ko} s))^2 < 0$ because $p_* (E\times_{\Ko} s)$ is supported on the exceptional curves of $(\pi_1)_s: Y_1\times_\Ko s \to \X_s$. It follows from \cref{prop: extremal} and its proof that the second step of the MMP for $(K_X + B + \epsilon E)\times_\Ko S$ over $S$ contracts $E_i \in \supp (p\times_\Ko id_S)_*(E\times_\Ko S)$ for some $i = 2, 3, \dots, n$ and preserves the triviality. By induction on the number of prime components for $E$, we conclude that each step of MMP contracts a prime component of $E\times_\Ko S$. In particular, the MMP terminates in finitely many steps. The outcome of the MMP is a projective elementary family $(Y_n, D_n)\times_\Ko S$ over $S$. Let
	$
	\pi_{MMP}:  X\times_{\Ko} S\to Y_n\times_\Ko S
	$
	be a morphism (over $S$) induced by the MMP. By construction,
	$
	\E(\pi) = \E(\pi_{MMP}) = \supp (E\times_\Ko S).
	$
	By \cref{prop: exc_1=exc_2}, the crepant birational map
	$
	\pi\circ \pi_{MMP}^{-1}: (Y_n, D_n)\times_\Ko S \dashrightarrow (\X, \B)
	$
	is an isomorphism over $S$. This completes the proof of \cref{main}.
	\end{proof}
	
	\section{Applications}
	Finally, we prove \cref{th: intro} by applying \cref{main_intro}.
	\begin{theorem}
		\label{th: final}
		The number of projective wlc klt models in a fixed $0$-class of log surfaces is finite up to log isomorphism.
	\end{theorem}
	\begin{proof}
		Let $(X_0,B_0)$ be a projective log surface, and let $\Cc$ be the class of all its projective wlc klt models $(X_\alpha,B_\alpha)$. Let $q_\alpha:(X_\alpha',B_\alpha')\to (X_\alpha,B_\alpha)$ be a terminalization morphism \cite[Corollary 1.4.3]{BCHM10}. By \cref{prop: unique_model}, there is a unique (up to log isomorphism) projective wlc trm log surface $(X,B)$ in the class $\Cc$. Therefore, we may assume $(X_\alpha',B_\alpha')=(X,B)$ for all $\alpha$. By \cite[Theorem~8.1]{Fuj12}, the $\mathbb{R}$-Cartier $\R$-divisor $K_X+B$ is semiample, that is, there is a projective contraction $p: X\to Y$ such that $K_X + B \sim_\R \sum_i r_i p^*H_i$ for some ample Cartier divisors $H_i$ on $Y$, and $r_i\in [0, 1]$.
		
		\textbf{Case 1:} $\kappa(K_X+B)=2$, that is, $p$ is a birational contraction. Then for every $E\in \E(q_\al)$ we have $(K_X + B)\cdot E = q_\al^*(K_{X_\al} + B_\al)\cdot E = 0$ by the projection formula. This implies that the natural map $p_\al: X_\al \dashrightarrow Y$ is a morphism, and $\E(q_\alpha)\subseteq \E(p)$ for all $\alpha$ according to \cref{prop: exc_1=exc_2}. Note that if $\E(q_\alpha)=\E(q_\beta)$ then $X_\alpha\cong X_\beta$ again by \cref{prop: exc_1=exc_2}. Moreover, we have $(X_\alpha,B_\alpha)\approx (X_\beta,B_\beta)$ provided that $(X_\al, B_\al)$ and $(X_\beta, B_\beta)$ are crepant birationally equivalent. Since $\E(p)\subset X$ is a closed subset of dimension 1, there are only finitely many possibilities for $\E(q_\al)\subseteq \E(p)$. Thus, the finiteness of the set $\E(p)$ implies that the set of (log) isomorphism classes in $\Cc$ is finite.
		
		\textbf{Case 2:} $\kappa(K_X+B)=1$, that is, $g$ is a projective fibration onto a smooth projective curve. Then for every $E\in \E(q_\al)$ we have $(K_X + B)\cdot E = q_\al^*(K_{X_\al} + B_\al)\cdot E = 0$ by the projection formula. This implies that there is a projective contraction $p_\al: X_\al \to Y$ according to the rigidity lemma (cf. \cite[Lemma 1.15]{Deb01}). Hence $\E(q_\alpha)\subseteq \E(X/Y):=\{\,C\subset X:\ \text{a contractible over $Y$ curve}\,\}$ for all~$\alpha$. Since $\dim X/Y = 1$, all contractible over $Y$ curves $C\subset X$ are included in the singular fibers of $p$. Hence, the set $\E(X/Y)$ is finite. If $\E(q_\al) = \E(q_\beta)$ then $(X_\al, B_\al)\cong (X_\beta, B_\beta)$ by \cref{prop: exc_1=exc_2}. Thus, the finiteness of the set $\E(X/Y)$ implies that the set of (log) isomorphism classes in $\Cc$ is finite.
		
		\textbf{Case 3:} $\kappa(K_X+B)=0$. Then $K_X+B\equiv 0$, and all $(X_\alpha,B_\alpha)$ are $0$-pairs. Consider the class $\Dd$ of all varieties $X_\alpha$ corresponding to the pairs $(X_\alpha,B_\alpha)$ in $\Cc$.
		
		\textbf{Case 3.1:} $B\neq 0$. Then the class $\Dd$ has bounded polarization by the following result due to Alexeev.
		\begin{theorem}{\cite[Theorem 6.9]{Ale94}}
			\label{th: Alexeev}
			Fix $\epsilon>0$. Consider the class $\Dd$ of projective normal surfaces $X_\alpha$ such that each $X_\alpha$ admits a boundary $B_\alpha$ satisfying:
			\begin{enumerate}
				\item The pair $(X_\alpha,B_\alpha)$ is an MR $\epsilon$-klt log pair.
				\item The $\mathbb{R}$-Cartier $\R$-divisor $-(K_{X_\alpha}+B_\alpha)$ is nef.
				\item The case, when $B_\alpha = 0, K_{X_\alpha}\equiv 0$, and $X_\alpha$ has Du Val singularities, is excluded.
			\end{enumerate}
			Then the class $\Dd$ has bounded polarization.
		\end{theorem}
		Condition (1) means that the inequalities in the definition of an $\epsilon$-klt log pair (\cref{def: sing}) need not hold for all log resolutions, but hold at least for the minimal resolution of $X_\alpha$. In our situation all log pairs $(X_\alpha,B_\alpha)$ satisfy this condition. Clearly, the contraction $q_\al: X\to X_\al$ factors through the minimal resolution for $X_\al$ because $X$ is smooth \cite[Theorem~4.5]{KM98}. Then $(X_\al, B_\al)$ is a MR $\epsilon$-klt log pair for any $0<\epsilon <  \min_i\{1-b_i\}$, where $B = \sum_i b_iB_i$. By the main \cref{main}, the set of (log) isomorphism classes in $\mathfrak{C}$ is finite.
		
		\textbf{Case 3.2:} $B=0$. Hence $K_X\equiv 0$. By \cite[Theorem 2.1]{Kaw97}, the surface $X$ admits only finitely many birational contractions up to automorphism.
	\end{proof}
	
	A key input in \cite[Theorem~2.1]{Kaw97} is Sterk’s cone theorem for K3 surfaces \cite{Ste85}, proved via the Torelli theorem. In future work we plan to develop an algebraic approach, which is independent of the Torelli Theorem, to bound polarization of Du Val K3 surfaces in a fixed birational class. Then, their finiteness up to isomorphism will  follow from  \cref{main}.
	
	By \cite[Theorem~2.14]{CL14}, the relative cone conjecture in dimension $\le d$ together with the MMP in dimension $d$ implies finiteness of minimal models up to isomorphism. Since the log version of the cone conjecture \cite{Tot10} is proved for log surfaces, \cref{th: final} is known to experts. However, the principal result of this paper is \cref{main} --- a method reducing the finiteness problem for models to boundedness of their polarization. A generalization of \cref{main} to higher dimensions will be treated in a forthcoming paper.

	\newpage
	\appendix
	\section{}\label{app}
	Suppose $h: \mathcal{V}\to S$ is a smooth projective morphism between non-singular varieties over $\Ko$. We write $S^\textnormal{an}$ for the analytification of $S$. Choose an ample over $S$ line bundle $\Li$ on $\mathcal{V}$, a closed point $s_0\in S$, and put $L = \Li_{s_0}$. Assume $d = \dim \mathcal{V}/S\ge 2$. Then the bilinear form $Q(x, y) = x\cdot y\cdot L^{d-2}_{s_0}$ is a Hodge-theoretic polarization on $V_\mathbb{Z} := H^2_\text{prim}(\mathcal{V}_{s_0}, \mathbb{Z}) = H^2_\text{prim}(\mathcal{V}_{s_0}, \mathbb{Q})\cap H^2(\mathcal{V}_{s_0}, \mathbb{Z})$. In the standard way \cite{Gri70}, we define a \textit{variation of integral polarized Hodge structures} of weight~2 associated with  $({R}^2h_*\mathbb{Z}_\mathcal{V})_\text{prim}\otimes \mathcal{O}_{S^\textnormal{an}}$ and $Q$. Let $\rho: \pi_1(S^\textnormal{an}, s_0)\to \text{Aut}(V_\mathbb{Z}, Q)$ be the monodromy representation. Then $\Gamma = \rho(\pi(S^\textnormal{an}, s_0))$ is \textit{the monodromy group}, and $\Gamma^a = \text{Aut}(V_\mathbb{Z}, Q)$ is \textit{the arithmetic monodromy group}. By $\Phi: S^\textnormal{an}\to \Gamma \backslash \mathbb{D}$ we denote the associated period map, where $\mathbb{D}$ is the period domain for integral polarized Hodge structures $({V}_\mathbb{Z}, F^\bullet {V}_\Ko, Q)$ of weight 2. This map is an analytic map between analytic spaces \cite[Lemma–Definition 4.6.3]{CMP17}. Similarly, the map $\Phi^a: S^\textnormal{an}\to \Gamma^a\backslash\mathbb{D}$ is analytic as well.  For every point $x\in {\Gamma}^a \backslash {\mathbb{D}}$ we choose a representative and denote by $V(x)$ its underlying unpolarized integral Hodge structure.
	\begin{lemma}
		\label{lem: finiteness_polarizations}
		Let $\Phi: S^\textnormal{an}\to \Gamma \backslash \mathbb{D}$ be the period map associated to the variation of integral polarized Hodge structures of weight 2 as above. Then the image of $\textnormal{Iso}_{\mathcal{V}_{s_0}}(\mathcal{V}/S)$ under $\Phi$ is finite.
	\end{lemma}
	\begin{proof}
	First, we prove that the image of $\textnormal{Iso}_{\mathcal{V}_{s_0}}(\mathcal{V}/S)$ under $\Phi^a$ is finite. By \cite[Corollary 2.3.5]{CMP17} the polarization $Q$ on $V_\mathbb{Z}$ can be extended to all of $\widehat{V}_\mathbb{Z}:= H^2(\mathcal{V}_{s_0}, \mathbb{Z})\cap H^2(\mathcal{V}_{s_0}, \mathbb{Q})$. In fact, there is an explicit formula for the extended polarization: $\widehat{Q}(x, y) = Q(x_0, y_0) - abL^d,$ where  $x = x_0 + aL, y = y_0 + bL$, and $x_0,y_0\in V_\mathbb{Z}$. Let $\widehat{\Phi}^a: S^\textnormal{an}\to  \widehat{\Gamma}^a \backslash \widehat{\mathbb{D}}$ be the period map associated with $(\text{R}^2h_*\mathbb{Z}_\mathcal{V}/\text{torsion})\otimes \mathcal{O}_{S^\textnormal{an}}$ and $\widehat{Q}$, where $\widehat{\Gamma}^a = \text{Aut}(\widehat{V}_\mathbb{Z}, \widehat{Q})$.  Set $x_0 = \widehat{\Phi}^a(s_0)$. Then $$\widehat{\Phi}^a(\textnormal{Iso}_{\mathcal{V}_{s_0}}(\mathcal{V}/S)) \subseteq \{x\in \widehat{\Gamma}^a \backslash \widehat{\mathbb{D}}: V(x)\approx V(x_0)\}.$$ By \cref{lem: finitenes_polarizations_orig} the latter set is finite. Let $(V(x_i), \widehat{Q})$ be integral polarized Hodge structures representing all points in $ \{x\in \widehat{\Gamma}^a \backslash \widehat{\mathbb{D}}: V(x)\approx V(x_0)\}$, and $i = 0, 1, \dots, I$. By $V^{(1,1)}_\R (x_i)$ we denote the real $(1,1)$-part of $V(x_i)$. Then for every index $i$ the set $\{L'\in V^{(1,1)}_\R (x_i)\cap \widehat{V}_{\mathbb{Z}}: \widehat{Q}(L', L') = -L^d\}$ is finite because the form $\widehat{Q}$ is negative-definite on $V^{(1,1)}_\R (x_i)$. Therefore, each $V(x_i)$ admits at most finitely many primitive decompositions with respect to $\widehat{Q}$. Let $V(x_{ij})$ be all (unpolarized) integral Hodge structures induced by $(V(x_i), \widehat{Q})$, and $j = 0, 1, \dots, J_i$. Recall that $Q = \widehat{Q}|_{V_\mathbb{Z}}$. This implies that $P = \Phi^a(\textnormal{Iso}_{\mathcal{V}_{s_0}}(\mathcal{V}/S))\subseteq \{x\in \Gamma^a\backslash\mathbb{D}: V(x)\approx V(x_{ij}) \text{ for some } i, j\}$. The latter set is finite due to \cref{lem: finitenes_polarizations_orig}. 
		
	Second, we prove the lemma. Due to \cite[Theorem 1.1]{BBT23}, the set $S':=(\Phi^a)^{-1}\left(P\right)$ is a closed algebraic subset of $S$, and $\Phi(S')$ is algebraic. By definition of the monodromy representation, $\Gamma$ is a subgroup of the (at most) countable group $\Gamma^a \subseteq \text{Aut}(V_\mathbb{Z})$. Hence the fibers of the quotient $\Gamma \backslash \mathbb{D} \to \Gamma^a \backslash \mathbb{D}$ are (at most) countable sets. From algebraicity of $I$ it follows that $\Phi(S')$ has (at most) finitely many connected components. Thus, the induced map $\Phi(S') \to \Phi^a(S')$ is finite. Therefore, $\Phi(S')$ is finite as well as $\Phi^a(S') = P$. This concludes the proof.
	\end{proof}
	Let us remark that the period domain $\mathbb{D}$ and the arithmetic monodromy group $\Gamma^a$ are defined for arbitrary polarized Hodge structures of any weight $n\in \mathbb{Z}_{\ge 0}$, which may not come from geometry.
	\begin{lemma}
		\label{lem: finitenes_polarizations_orig}
		Let ${\mathbb{D}}$ be the period domain for polarized integral Hodge structures $({V}_\mathbb{Z}, F^\bullet {V}_\Ko, Q)$  of weight $n\in \mathbb{Z}_{\ge0}$. Then for any $x_0\in {\Gamma}^a \backslash {\mathbb{D}}$ the set $\{x\in {\Gamma}^a \backslash {\mathbb{D}}: V(x)\approx V(x_0)\}$ is finite, where $\Gamma^a = \textnormal{Aut}(V_\mathbb{Z}, Q)$ is the arithmetic monodromy group.
	\end{lemma}
	\begin{proof}
		The proof of this lemma is taken verbatim from \cite[Lemma 2.22]{BFMT25}. In \cite{NN81}, this lemma is proven for abelian varieties, and the same proof works in this more general context. 
		Next, consider the algebra $B=\text{End}(V(x_0))$ of unpolarized Hodge endomorphisms. Since polarizable Hodge structures are a semisimple category, it follows that $B_{\mathbb{Q}}$ is a semisimple algebra over $\mathbb{Q}$. Moreover, the polarization gives an involution $\theta$ of $B_\mathbb{Q}$; let $ B^{\theta}\subset B_\mathbb{Q}$ be the elements fixed by $\theta$. Letting $G$ be the algebraic group $B^{\times}$, we see that $\theta$ gives an involution $G\to G^{\text{op}}$. We shall prove that $V(x_0)$ admits only finitely many orbits of polarizations of discriminant $\mathrm{disc}({Q})$ for the action of $G(\mathbb{Z})$, which will prove the lemma.
		
		Let $P$ be the set of polarizations of $V(x_0)$. There is a natural injection $\iota:P\to B_\mathbb{Q}$
		given by $\iota(Q'):=  \phi^{-1}_{{Q}}\circ \phi_{Q'}$ where $\phi_{Q'}:V(x_0)\to V(x_0)^{\vee}$. Let $F\subset B^{\theta}$ be the minimal lattice containing $\iota(P)$. There is a natural action of $G$ on $B^{\theta}$ given by $\pi(g)s: \theta(g^{-1})\circ s\circ g^{-1}$ for which $\iota$ is equivariant, and for which $G(\mathbb{Z})$ preserves $F$. Finally, let $F_1:= \{f\in B^{\theta}_{\mathbb{C}}\mid \det(f)=1\}$. As in \cite[Lemma 3.1]{NN81} (this lemma uses only that $(B_\mathbb{Q},\theta)$ is an involutive semisimple algebra) the orbits of $G_\mathbb{C}$ on $F_1$ are finite in number and closed. The result now follows by \cite[Thm 6.9]{BH62}.
	\end{proof}
	\newpage
	\bibliographystyle{plain}

\end{document}